\documentclass[12pt,fleqn]{article}

\usepackage{amsmath,amssymb,amsthm}
\usepackage{geometry} 
\usepackage[pdfpagemode=UseNone,pdfstartview=FitH]{hyperref}
\usepackage{color,tikz,float}

\newcommand{\As}[1][]{A_s #1}
\newcommand{\Bs}[1][]{B_s #1}
\newcommand{\bdry}[1]{\partial #1}
\newcommand{\bgset}[1]{\big\{#1\big\}}
\newcommand{\dint}{\ds{\int}}
\newcommand{\ds}[1]{\displaystyle #1}
\newcommand{\eps}{\varepsilon}
\newcommand{\F}{{\mathcal F}}
\newcommand{\M}{{\mathcal M}}
\newcommand{\N}{\mathcal N}
\newcommand{\norm}[2][]{\left\|#2\right\|_{#1}}
\renewcommand{\O}{\text{O}}
\renewcommand{\o}{\text{o}}
\newcommand{\PS}[1]{$(\text{PS})_{#1}$}
\newcommand{\QED}{\mbox{\qedhere}}
\newcommand{\R}{\mathbb R}
\newcommand{\restr}[2]{\left.#1\right|_{#2}}
\newcommand{\seq}[1]{\left(#1\right)}
\newcommand{\set}[1]{\left\{#1\right\}}
\newcommand{\wto}{\rightharpoonup}
\newcommand{\Z}{\mathbb Z}

\newenvironment{enumroman}{\begin{enumerate}

}{\end{enumerate}}

\newtheorem{lemma}{Lemma}[section]
\newtheorem{proposition}[lemma]{Proposition}
\newtheorem{theorem}[lemma]{Theorem}

\theoremstyle{definition}
\newtheorem{definition}[lemma]{Definition}

\theoremstyle{remark}
\newtheorem{remark}[lemma]{Remark}

\numberwithin{equation}{section}

\title{\bf Prescribed energy solutions to some scaled problems via a scaled Nehari manifold}
\author{\bf Kanishka Perera\\
Department of Mathematics\\
Florida Institute of Technology\\
150 W University Blvd, Melbourne, FL 32901-6975, USA\\
\em kperera@fit.edu\\
[\medskipamount]
\bf Kaye Silva\\
Instituto de Matem\'{a}tica e Estat\'{i}stica\\
Universidade Federal de Goi\'{a}s\\
Rua Samambaia, 74001-970 Goi\^{a}nia, GO, Brazil\\
\em kayesilva@ufg.br}
\date{}

\begin{document}

\maketitle

\begin{abstract}
We prove the existence, multiplicity, and bifurcation of solutions with prescribed energy for a broad class of scaled problems by introducing a suitable notion of scaling based Nehari manifold. Applications are given to Schr\"{o}dinger--Poisson--Slater type equations.
\end{abstract}

\begin{center}
	\begin{minipage}{12cm}
		\tableofcontents
	\end{minipage}
\end{center}

\newpage

\section{Introduction}

Consider the Schr\"{o}dinger--Poisson--Slater type equation
\begin{equation} \label{101}
- \Delta u + \left(\frac{1}{4 \pi |x|} \star u^2\right) u = f(u) \quad \text{in } \R^3,
\end{equation}
where $f : \R \to \R$ is a continuous function satisfying a suitable growth condition. This equation is related to the Thomas--Fermi--Dirac--von\;Weizs\"acker model in Density Functional Theory (DFT), where the local nonlinearity $f$ takes the form $f(u) = u^{5/3} - u^{7/3}$ (see, e.g., Lieb \cite{MR629207,MR641371}, Le Bris and Lions \cite{MR2149087}, Lu and Otto \cite{MR3251907}, Frank et al.\! \cite{MR3762278}, and references therein). More general local nonlinearities also arise in DFT and quantum chemistry models such as Kohn-Sham's, where $f$ is the so-called exchange-correlation potential and is not explicitly known (see, e.g., Anantharaman and Canc\`es \cite{MR2569902} and references therein). The purpose of this paper is to study the existence, multiplicity, and bifurcation of solutions with prescribed energy for a broad class of local nonlinearities that include these physical models as special cases.

The natural space to look for solutions to the equation \eqref{101} is the Coulomb-Sobolev space
\[
E(\R^3) = \set{u \in D^{1,2}(\R^3) : \int_{\R^3} \int_{\R^3} \frac{u^2(x)\, u^2(y)}{|x - y|}\, dx\, dy < \infty}
\]
endowed with the norm
\[
\norm{u} = \left[\int_{\R^3} |\nabla u|^2\, dx + \left(\int_{\R^3} \int_{\R^3} \frac{u^2(x)\, u^2(y)}{|x - y|}\, dx\, dy\right)^{1/2}\right]^{1/2}
\]
(see Lions \cite{MR636734} and Ruiz \cite{MR2679375}). We work in the subspace $E_r(\R^3)$ of radial functions in $E(\R^3)$. It was shown in Ruiz \cite[Theorem 1.2]{MR2679375} that $E_r(\R^3)$ is embedded in $L^\rho(\R^3)$ continuously for $\rho \in (18/7,6]$ and compactly for $\rho \in (18/7,6)$, and these ranges are optimal by Mercuri et al. \cite[Theorem 4 and Theorem 5]{MR3568051} (see also Bellazzini et al. \cite{MR3852465}). So a natural growth condition for the local nonlinearity $f$ is
\begin{equation} \label{53}
|f(t)| \le a_1\, |t|^{\sigma - 1} + a_2\, |t|^{\tau - 1} \quad \forall t \in \R
\end{equation}
for some constants $a_1, a_2 > 0$ and $18/7 < \sigma < \tau \le 6$. Under this growth assumption, the associated energy functional
\[
\Phi(u) = \frac{1}{2} \int_{\R^3} |\nabla u|^2\, dx + \frac{1}{16 \pi} \int_{\R^3} \int_{\R^3} \frac{u^2(x)\, u^2(y)}{|x - y|}\, dx\, dy - \int_{\R^3} F(u)\, dx, \quad u \in E_r(\R^3),
\]
where $F(t) = \int_0^t f(\tau)\, d\tau$ is the primitive of $f$, is well-defined, of class $C^1$, and its critical points are the solutions to equation \eqref{101}. Given $c \in \R$, we look for solutions $u \in E_r(\R^3)$ such that $\Phi(u) = c$.

The nonlinear eigenvalue problem
\begin{equation} \label{55}
- \Delta u + \left(\frac{1}{4 \pi |x|} \star u^2\right) u = \lambda\, |u| u \quad \text{in } \R^3
\end{equation}
will play a central role in our results for the equation \eqref{101}. This eigenvalue problem is not only nonlinear, but also nonhomogeneous in the sense that the three terms appearing in it contain different powers of $u$. However, it has the following scale invariance as noted in \cite{MR2902293}. Consider the scaling
\[
u_t(x) = t^2\, u(tx), \quad (x,t) \in \R^3 \times [0,\infty).
\]
If $u$ is a solution of the equation \eqref{55}, then the entire $1$-parameter family of functions $\set{u_t : t \in \R}$ are solutions. So we will refer to it as a scaled eigenvalue problem. As noted in Mercuri and Perera \cite{MePe2}, this eigenvalue problem leads to a natural classification of the local nonlinearity $f$ in equation \eqref{101} in terms of its scaling properties as follows:
\begin{enumroman}
\item $f$ is subscaled if $\ds{\lim_{|t| \to \infty}\, \frac{f(t)}{|t|\, t}} = 0$,
\item $f$ is asymptotically scaled if $\ds{\lim_{|t| \to \infty}\, \frac{f(t)}{|t|\, t}} = \lambda$ for some $\lambda \in \R \setminus \set{0}$,
\item $f$ is superscaled if $\ds{\lim_{|t| \to \infty}\, \frac{f(t)}{|t|\, t}} = \infty$.
\end{enumroman}

We write $f$ in the form $f(t) = \lambda\, |t|\, t + g(t)$, where $g$ is subscaled, superscaled, or a combination of both with opposite signs, and consider the equation
\begin{equation} \label{102}
- \Delta u + \left(\frac{1}{4 \pi |x|} \star u^2\right) u = \lambda\, |u| u + g(u) \quad \text{in } \R^3.
\end{equation}
The associated energy functional is
\begin{multline*}
\Phi_\lambda(u) = \frac{1}{2} \int_{\R^3} |\nabla u|^2\, dx + \frac{1}{16 \pi} \int_{\R^3} \int_{\R^3} \frac{u^2(x)\, u^2(y)}{|x - y|}\, dx\, dy - \frac{\lambda}{3} \int_{\R^3} |u|^3\, dx - \int_{\R^3} G(u)\, dx,\\[7.5pt]
u \in E_r(\R^3),
\end{multline*}
where $G(t) = \int_0^t g(\tau)\, d\tau$. Given $c \in \R$, we are interested in finding pairs $(\lambda,u) \in \R \times E_r(\R^3)$ such that $\Phi_\lambda'(u) = 0$ and $\Phi_\lambda(u) = c$. Let
\begin{equation} \label{200}
I(u) = \frac{1}{2} \int_{\R^3} |\nabla u|^2\, dx + \frac{1}{16 \pi} \int_{\R^3} \int_{\R^3} \frac{u^2(x)\, u^2(y)}{|x - y|}\, dx\, dy, \qquad J(u) = \frac{1}{3} \int_{\R^3} |u|^3\, dx.
\end{equation}
For $u \ne 0$, the equation $\Phi_\lambda(u) = c$ has the unique solution for $\lambda$ given by
\[
\lambda_c(u) = \frac{I(u) - \dint_{\R^3} G(u)\, dx - c}{J(u)}, \quad u \in E_r(\R^3) \setminus \set{0}.
\]
For $c \in \R$, $u \in E_r(\R^3) \setminus \set{0}$ is a critical point of $\Phi_\lambda$ with critical value $c$ if and only if it is a critical point of $\lambda_c$ with critical value $\lambda$ (see Proposition \ref{Proposition 7}). We will use this observation to study solutions of the equation \eqref{102} with energy $c$.

We will use a Nehari type manifold to study critical points of the functional $\lambda_c$. However, the standard Nehari manifold (see, e.g., Szulkin and Weth \cite{MR2768820}) is not suitable for this purpose because of the form of $\lambda_c$. Therefore we introduce a new manifold based on the following scaling property of $I$ and $J$:
\[
I(u_t) = t^3 I(u), \quad J(u_t) = t^3 J(u) \quad \forall u \in E_r(\R^3),\, t \ge 0.
\]
Let $\M = \set{u \in E_r(\R^3) : I(u) = 1}$. Then $E_r(\R^3)$ can be parametrized as $\set{u_t : u \in \M,\, t \ge 0}$ (see \eqref{11}). Set
\[
\varphi_{c,u}(t) = \lambda_c(u_t), \quad u \in \M,\, t > 0
\]
and let
\[
\N_c = \set{u_t : u \in \M,\, t > 0,\, \varphi_{c,u}'(t) = 0}.
\]
We will show that, under suitable assumptions on $g$, $\N_c$ is a $C^1$-Finsler manifold and a natural constraint for $\lambda_c$ (see Lemma \ref{Lemma 8} and Lemma \ref{Lemma 6}). We will refer to it as the scaled Nehari manifold. We will study critical points of $\restr{\lambda_c}{\N_c}$ using a variational framework for scaled functionals recently developed in Mercuri and Perera \cite{MePe2}. In what follows, $\seq{\lambda_k}$ is the sequence of eigenvalues of the eigenvalue problem \eqref{55} given in Theorem \ref{Theorem 7} in the next section.

First we consider the subscaled case $g(t) = |t|^{\sigma - 2}\, t$, where $18/7 < \sigma < 3$, corresponding to the equation
\begin{equation} \label{103}
- \Delta u + \left(\frac{1}{4 \pi |x|} \star u^2\right) u = \lambda\, |u| u + |u|^{\sigma - 2}\, u \quad \text{in } \R^3.
\end{equation}
Here the energy functional is
\[
\Phi_\lambda(u) = I(u) - \lambda\, J(u) - \frac{1}{\sigma} \int_{\R^3} |u|^\sigma\, dx, \quad u \in E_r(\R^3)
\]
and
\[
\lambda_c(u) = \frac{I(u) - \dfrac{1}{\sigma} \dint_{\R^3} |u|^\sigma\, dx - c}{J(u)}, \quad u \in E_r(\R^3) \setminus \set{0}.
\]
We have the following theorem.

\begin{theorem} \label{Theorem 101}
Let $18/7 < \sigma < 3$.
\begin{enumroman}
\item Equation \eqref{103} has no nontrivial solution with energy $c \ge 0$. For each $c \in (- \infty,0)$, $\lambda_c$ has a sequence of critical values $\lambda_{c,k} \nearrow \infty$, and equation \eqref{103} with $\lambda = \lambda_{c,k}$ has a nontrivial solution $u_{c,k}$ with $\Phi_{\lambda_{c,k}}(u_{c,k}) = c$.
\item For each $\lambda \in \R$, equation \eqref{103} has a sequence of nontrivial solutions $\seq{v_{\lambda,k}}$ such that $\Phi_\lambda(v_{\lambda,k}) \nearrow 0$ and $\norm{v_{\lambda,k}} \to 0$, so $(\lambda,0) \in \R \times E_r(\R^3)$ is a bifurcation point for \eqref{103}.
\end{enumroman}
\end{theorem}

Next we consider the subscaled case $g(t) = - |t|^{\sigma - 2}\, t$, where $18/7 < \sigma < 3$, corresponding to the equation
\begin{equation} \label{104}
- \Delta u + \left(\frac{1}{4 \pi |x|} \star u^2\right) u = \lambda\, |u| u - |u|^{\sigma - 2}\, u \quad \text{in } \R^3.
\end{equation}
Here the energy functional is
\[
\Phi_\lambda(u) = I(u) - \lambda\, J(u) + \frac{1}{\sigma} \int_{\R^3} |u|^\sigma\, dx, \quad u \in E_r(\R^3)
\]
and
\[
\lambda_c(u) = \frac{I(u) + \dfrac{1}{\sigma} \dint_{\R^3} |u|^\sigma\, dx - c}{J(u)}, \quad u \in E_r(\R^3) \setminus \set{0}.
\]
We have the following theorem.

\begin{theorem} \label{Theorem 102}
Let $18/7 < \sigma < 3$.
\begin{enumroman}
\item Equation \eqref{104} has no nontrivial solution with energy $c \le 0$. For each $c \in (0,\infty)$, $\lambda_c$ has a sequence of critical values $\lambda_{c,k} \nearrow \infty$, and equation \eqref{104} with $\lambda = \lambda_{c,k}$ has a nontrivial solution $u_{c,k}$ with $\Phi_{\lambda_{c,k}}(u_{c,k}) = c$.
\item For $\lambda \le \lambda_1$, equation \eqref{104} has no nontrivial solution. For each $\lambda > \lambda_k$, equation \eqref{104} has at least $k$ pairs of nontrivial solutions.
\end{enumroman}
\end{theorem}

Next we consider the superscaled case $g(t) = |t|^{\tau - 2}\, t$, where $3 < \tau < 6$, corresponding to the equation
\begin{equation} \label{105}
- \Delta u + \left(\frac{1}{4 \pi |x|} \star u^2\right) u = \lambda\, |u| u + |u|^{\tau - 2}\, u \quad \text{in } \R^3.
\end{equation}
Here the energy functional is
\[
\Phi_\lambda(u) = I(u) - \lambda\, J(u) - \frac{1}{\tau} \int_{\R^3} |u|^\tau\, dx, \quad u \in E_r(\R^3)
\]
and
\[
\lambda_c(u) = \frac{I(u) - \dfrac{1}{\tau} \dint_{\R^3} |u|^\tau\, dx - c}{J(u)}, \quad u \in E_r(\R^3) \setminus \set{0}.
\]
We have the following theorem.

\begin{theorem} \label{Theorem 103}
Let $3 < \tau < 6$.
\begin{enumroman}
\item Equation \eqref{105} has no nontrivial solution with energy $c \le 0$. For each $c \in (0,\infty)$, $\lambda_c$ has a sequence of critical values $\lambda_{c,k} \nearrow \infty$, and equation \eqref{105} with $\lambda = \lambda_{c,k}$ has a nontrivial solution $u_{c,k}$ with $\Phi_{\lambda_{c,k}}(u_{c,k}) = c$.
\item For each $\lambda \in \R$, equation \eqref{105} has a sequence of nontrivial solutions $\seq{v_{\lambda,k}}$ such that $\Phi_\lambda(v_{\lambda,k}) \nearrow \infty$ and $\norm{v_{\lambda,k}} \to \infty$, so $(\lambda,\infty) \in \R \times E_r(\R^3)$ is a bifurcation point for \eqref{105}.
\end{enumroman}
\end{theorem}

Next we consider the superscaled case $g(t) = - |t|^{\tau - 2}\, t$, where $3 < \tau < 6$, corresponding to the equation
\begin{equation} \label{106}
- \Delta u + \left(\frac{1}{4 \pi |x|} \star u^2\right) u = \lambda\, |u| u - |u|^{\tau - 2}\, u \quad \text{in } \R^3.
\end{equation}
Here the energy functional is
\[
\Phi_\lambda(u) = I(u) - \lambda\, J(u) + \frac{1}{\tau} \int_{\R^3} |u|^\tau\, dx, \quad u \in E_r(\R^3)
\]
and
\[
\lambda_c(u) = \frac{I(u) + \dfrac{1}{\tau} \dint_{\R^3} |u|^\tau\, dx - c}{J(u)}, \quad u \in E_r(\R^3) \setminus \set{0}.
\]
We have the following theorem.

\begin{theorem} \label{Theorem 104}
Let $3 < \tau < 6$.
\begin{enumroman}
\item Equation \eqref{106} has no nontrivial solution with energy $c \ge 0$. For each $c \in (- \infty,0)$, $\lambda_c$ has a sequence of critical values $\lambda_{c,k} \nearrow \infty$, and equation \eqref{106} with $\lambda = \lambda_{c,k}$ has a nontrivial solution $u_{c,k}$ with $\Phi_{\lambda_{c,k}}(u_{c,k}) = c$.
\item For $\lambda \le \lambda_1$, equation \eqref{106} has no nontrivial solution. For each $\lambda > \lambda_k$, equation \eqref{106} has at least $k$ pairs of nontrivial solutions.
\end{enumroman}
\end{theorem}

Next we consider the mixed sub-superscaled case $g(t) = |t|^{\sigma - 2}\, t - |t|^{\tau - 2}\, t$, where $18/7 < \sigma < 3 < \tau < 6$, corresponding to the equation
\begin{equation} \label{107}
- \Delta u + \left(\frac{1}{4 \pi |x|} \star u^2\right) u = \lambda\, |u| u + |u|^{\sigma - 2}\, u - |u|^{\tau - 2}\, u \quad \text{in } \R^3.
\end{equation}
Here the energy functional is
\[
\Phi_\lambda(u) = I(u) - \lambda\, J(u) - \frac{1}{\sigma} \int_{\R^3} |u|^\sigma\, dx + \frac{1}{\tau} \int_{\R^3} |u|^\tau\, dx, \quad u \in E_r(\R^3)
\]
and
\[
\lambda_c(u) = \frac{I(u) - \dfrac{1}{\sigma} \dint_{\R^3} |u|^\sigma\, dx + \dfrac{1}{\tau} \dint_{\R^3} |u|^\tau\, dx - c}{J(u)}, \quad u \in E_r(\R^3) \setminus \set{0}.
\]
We have the following theorem.

\begin{theorem} \label{Theorem 105}
Let $18/7 < \sigma < 3 < \tau < 6$.
\begin{enumroman}
\item Equation \eqref{107} has no nontrivial solution with energy $c \ge 0$. For each $c \in (- \infty,0)$, $\lambda_c$ has a sequence of critical values $\lambda_{c,k} \nearrow \infty$, and equation \eqref{107} with $\lambda = \lambda_{c,k}$ has a nontrivial solution $u_{c,k}$ with $\Phi_{\lambda_{c,k}}(u_{c,k}) = c$.
\item For each $\lambda \in \R$, equation \eqref{107} has a sequence of nontrivial solutions $\seq{v_{\lambda,k}}$ such that $\Phi_\lambda(v_{\lambda,k}) \nearrow 0$ and $\norm{v_{\lambda,k}} \to 0$, so $(\lambda,0) \in \R \times E_r(\R^3)$ is a bifurcation point for \eqref{107}.
\end{enumroman}
\end{theorem}

Finally we consider the mixed sub-superscaled case $g(t) = - |t|^{\sigma - 2}\, t + |t|^{\tau - 2}\, t$, where $18/7 < \sigma < 3 < \tau < 6$, corresponding to the equation
\begin{equation} \label{108}
- \Delta u + \left(\frac{1}{4 \pi |x|} \star u^2\right) u = \lambda\, |u| u - |u|^{\sigma - 2}\, u +|u|^{\tau - 2}\, u \quad \text{in } \R^3.
\end{equation}
Here the energy functional is
\[
\Phi_\lambda(u) = I(u) - \lambda\, J(u) + \frac{1}{\sigma} \int_{\R^3} |u|^\sigma\, dx - \frac{1}{\tau} \int_{\R^3} |u|^\tau\, dx, \quad u \in E_r(\R^3)
\]
and
\[
\lambda_c(u) = \frac{I(u) + \dfrac{1}{\sigma} \dint_{\R^3} |u|^\sigma\, dx - \dfrac{1}{\tau} \dint_{\R^3} |u|^\tau\, dx - c}{J(u)}, \quad u \in E_r(\R^3) \setminus \set{0}.
\]
We have the following theorem.

\begin{theorem} \label{Theorem 106}
Let $18/7 < \sigma < 3 < \tau < 6$.
\begin{enumroman}
\item Equation \eqref{108} has no nontrivial solution with energy $c \le 0$. For each $c \in (0,\infty)$, $\lambda_c$ has a sequence of critical values $\lambda_{c,k} \nearrow \infty$, and equation \eqref{108} with $\lambda = \lambda_{c,k}$ has a nontrivial solution $u_{c,k}$ with $\Phi_{\lambda_{c,k}}(u_{c,k}) = c$.
\item For each $\lambda \in \R$, equation \eqref{108} has a sequence of nontrivial solutions $\seq{v_{\lambda,k}}$ such that $\Phi_\lambda(v_{\lambda,k}) \nearrow \infty$ and $\norm{v_{\lambda,k}} \to \infty$, so $(\lambda,\infty) \in \R \times E_r(\R^3)$ is a bifurcation point for \eqref{108}.
\end{enumroman}
\end{theorem}

In the next section we will prove these theorems in an abstract setting involving scaled operators and scaled eigenvalue problems that was introduced in Mercuri and Perera \cite{MePe2}. We will apply our abstract results to the equation \eqref{102} and deduce Theorems \ref{Theorem 101}--\ref{Theorem 106} in Section \ref{Section 3}.

Our abstract results have other applications. For example, they can be used to establish new results analogous to those in Theorems \ref{Theorem 101}--\ref{Theorem 106} for the semilinear elliptic boundary value problem
\[
\left\{\begin{aligned}
- \Delta u & = \lambda u + \mu\, |u|^{\sigma - 2}\, u + \nu\, |u|^{\tau - 2}\, u && \text{in } \Omega\\[10pt]
u & = 0 && \text{on } \bdry{\Omega},
\end{aligned}\right.
\]
where $\Omega$ is a bounded domain in $\R^N,\, N \ge 2$, $1 < \sigma < 2 < \tau < 2N/(N - 2)$, $\lambda \in \R$ is a parameter, and $\mu \nu \le 0$. Here the scaling is the standard one
\[
H^1_0(\Omega) \times [0,\infty) \to H^1_0(\Omega), \quad (u,t) \mapsto tu.
\]
Other results on prescribed energy solutions to this problem can be found in Ramos Quoirin et al.\! \cite{MR4736027}. We also note that in the work Leite et al. \cite{LeQuSi}, the prescribed energy solutions problem was considered for more general equations under the standard scaling.

As another example, our abstract results can be applied to the nonlocal problem
\[
\left\{\begin{aligned}
(- \Delta)^s\, u & = \lambda u + \mu\, |u|^{\sigma - 2}\, u + \nu\, |u|^{\tau - 2}\, u && \text{in } \Omega\\[10pt]
u & = 0 && \text{in } \R^N \setminus \Omega,
\end{aligned}\right.
\]
where $\Omega$ is a bounded domain in $\R^N$ with Lipschitz boundary, $(- \Delta)^s$ is the fractional Laplacian operator defined on smooth functions by
\[
(- \Delta)^s\, u(x) = 2 \lim_{\eps \searrow 0} \int_{\R^N \setminus B_\eps(x)} \frac{u(x) - u(y)}{|x - y|^{N+2s}}\, dy, \quad x \in \R^N,
\]
$s \in (0,N/2]$, $1 < \sigma < 2 < \tau < 2N/(N - 2s)$, $\lambda \in \R$ is a parameter, and $\mu \nu \le 0$. More generally, we can consider applications to problems involving the $p$-Laplacian or the fractional $p$-Laplacian operators, or superposition operators of mixed fractional order as in Dipierro et al.\! \cite{DiPeSpVa2,MR4736013}.

\section{Abstract results}

In this section we prove our abstract results on prescribed energy solutions of scaled problems. We begin by recalling some results on scaled eigenvalue problems recently proved in Mercuri and Perera \cite{MePe2}. Then we develop our theory of the scaled Nehari manifold for solutions with prescribed energy. Our main existence, multiplicity, and bifurcation results are given at the end of the section.

\subsection{Scaled eigenvalue problems}

Let $W$ be a reflexive Banach space. The following notion of a scaling on $W$ was introduced in Mercuri and Perera \cite{MePe2}.

\begin{definition}[{\cite[Definition 1.1]{MePe2}}]
A scaling on $W$ is a continuous mapping
\[
W \times [0,\infty) \to W, \quad (u,t) \mapsto u_t
\]
satisfying
\begin{enumerate}
\item[$(A_1)$] $(u_{t_1})_{t_2} = u_{t_1 t_2}$ for all $u \in W$ and $t_1, t_2 \ge 0$,
\item[$(A_2)$] $(\tau u)_t = \tau u_t$ for all $u \in W$, $\tau \in \R$, and $t \ge 0$,
\item[$(A_3)$] $u_0 = 0$ and $u_1 = u$ for all $u \in W$,
\item[$(A_4)$] $u_t$ is bounded on bounded sets in $W \times [0,\infty)$,
\item[$(A_5)$] $\exists s > 0$ such that $\norm{u_t} = \O(t^s)$ as $t \to \infty$, uniformly in $u$ on bounded sets.
\end{enumerate}
\end{definition}

Denote by $W^\ast$ the dual of $W$. Recall that $q \in C(W,W^\ast)$ is a potential operator if there is a functional $Q \in C^1(W,\R)$, called a potential for $q$, with Fr\'{e}chet derivative $Q' = q$. By replacing $Q$ with $Q - Q(0)$ if necessary, we may assume that $Q(0) = 0$.

\begin{definition}[{\cite[Definition 2.1]{MePe2}}]
A scaled operator is an odd potential operator $\As \in C(W,W^\ast)$ that maps bounded sets into bounded sets and satisfies
\[
\As(u_t) v_t = t^s \As(u) v \quad \forall u, v \in W,\, t \ge 0.
\]
\end{definition}

Let $\As$ and $\Bs$ be scaled operators satisfying
\begin{enumerate}
\item[$(A_6)$] $\As(u) u > 0$ for all $u \in W \setminus \set{0}$,
\item[$(A_7)$] every sequence $\seq{u_j}$ in $W$ such that $u_j \wto u$ and $\As(u_j)(u_j - u) \to 0$ has a subsequence that converges strongly to $u$,
\item[$(A_8)$] $\Bs(u) u > 0$ for all $u \in W \setminus \set{0}$,
\item[$(A_9)$] if $u_j \wto u$ in $W$, then $\Bs(u_j) \to \Bs(u)$ in $W^\ast$.
\end{enumerate}
The scaled eigenvalue problem
\begin{equation} \label{1}
\As(u) = \lambda \Bs(u) \quad \text{in } W^\ast
\end{equation}
was studied in \cite{MePe2}. We say that $\lambda \in \R$ is an eigenvalue of this problem if there is a $u \in W \setminus \set{0}$, called an eigenfunction associated with $\lambda$, satisfying \eqref{1}. Then $u_t$ is also an eigenfunction associated with $\lambda$ for any $t > 0$ since
\[
\As(u_t) v = \As(u_t) (v_{t^{-1}})_t = t^s \As(u) v_{t^{-1}} = t^s \lambda \Bs(u) v_{t^{-1}} = \lambda \Bs(u_t) (v_{t^{-1}})_t = \lambda \Bs(u_t) v
\]
for all $v \in W$.

The potentials
\begin{equation} \label{2}
I_s(u) = \int_0^1 \As(\tau u) u\, d\tau, \quad J_s(u) = \int_0^1 \Bs(\tau u) u\, d\tau, \quad u \in W
\end{equation}
of $\As$ and $\Bs$, respectively, are even, \hspace{-5pt} bounded \hspace{-5pt} on \hspace{-5pt} bounded sets, and have the scaling property
\begin{equation} \label{3}
I_s(u_t) = t^s I_s(u), \quad J_s(u_t) = t^s J_s(u) \quad \forall u \in W,\, t \ge 0
\end{equation}
(see \cite[Proposition 2.2]{MePe2}). By \eqref{2}, $(A_6)$, and $(A_8)$,
\begin{equation} \label{4}
I_s(u) > 0, \quad J_s(u) > 0 \quad \forall u \in W \setminus \set{0}.
\end{equation}
It follows from $(A_9)$ that if $u_j \wto u$ in $W$, then $J_s(u_j) \to J_s(u)$ (see {\cite[Proposition 2.4]{MePe2}}). We assume that $I_s$ and $J_s$ satisfy
\begin{enumerate}
\item[$(A_{10})$] $I_s$ is coercive, i.e., $I_s(u) \to \infty$ as $\norm{u} \to \infty$,
\item[$(A_{11})$] every solution of problem \eqref{1} satisfies $I_s(u) = \lambda\, J_s(u)$.
\end{enumerate}

We have $I_s'(0) = \As(0) = 0$ since $\As$ is odd, so the origin is a critical point of $I_s$. It is the only critical point of $I_s$ since
\[
I_s'(u) u = \As(u) u > 0 \quad \forall u \in W \setminus \set{0}
\]
by $(A_6)$. So $I_s(0) = 0$ is the only critical value of $I_s$, and hence it follows from the implicit function theorem that
\[
\M_s = \bgset{u \in W : I_s(u) = 1}
\]
is a $C^1$-Finsler manifold. Since $I_s$ is continuous, even, and coercive, $\M_s$ is complete, symmetric, and bounded. Let
\[
\Psi(u) = \frac{1}{J_s(u)}, \quad u \in W \setminus \set{0}
\]
and let $\widetilde{\Psi} = \restr{\Psi}{\M_s}$. Then eigenvalues of problem \eqref{1} coincide with critical values of $\widetilde{\Psi}$ (see \cite[Proposition 2.5]{MePe2}).

Let $\F$ denote the class of symmetric subsets of $\M_s$ and let $i(M)$ denote the $\Z_2$-cohomological index of $M \in \F$ (see Fadell and Rabinowitz \cite{MR0478189}). For $k \ge 1$, let $\F_k = \bgset{M \in \F : i(M) \ge k}$ and set
\begin{equation} \label{5}
\lambda_k := \inf_{M \in \F_k}\, \sup_{u \in M}\, \widetilde{\Psi}(u).
\end{equation}
The following theorem was proved in \cite{MePe2}.

\begin{theorem}[{\cite[Theorem 2.10]{MePe2}}] \label{Theorem 7}
Assume $(A_1)$--$(A_{11})$. Then $\lambda_k \nearrow \infty$ is a sequence of eigenvalues of problem \eqref{1}.
\begin{enumroman}
\item The first eigenvalue is given by
    \[
    \lambda_1 = \min_{u \in \M_s}\, \widetilde{\Psi}(u) > 0.
    \]
\item If $\lambda_k = \dotsb = \lambda_{k+m-1} = \lambda$ and $E_\lambda$ is the set of eigenfunctions associated with $\lambda$ that lie on $\M_s$, then $i(E_\lambda) \ge m$.
\item If $\lambda_k < \lambda < \lambda_{k+1}$, then
    \[
    i(\widetilde{\Psi}^{\lambda_k}) = i(\M_s \setminus \widetilde{\Psi}_\lambda) = i(\widetilde{\Psi}^\lambda) = i(\M_s \setminus \widetilde{\Psi}_{\lambda_{k+1}}) = k,
    \]
    where $\widetilde{\Psi}^a = \bgset{u \in \M_s : \widetilde{\Psi}(u) \le a}$ and $\widetilde{\Psi}_a = \bgset{u \in \M_s : \widetilde{\Psi}(u) \ge a}$ for $a \in \R$.
\end{enumroman}
\end{theorem}

\subsection{Formulation of the problem}

We consider the question of existence and multiplicity of solutions to the nonlinear operator equation
\begin{equation} \label{18}
\As(u) = \lambda \Bs(u) + f(u) + g(u) \quad \text{in } W^\ast,
\end{equation}
where $f, g \in C(W,W^\ast)$ are odd potential operators satisfying
\begin{enumerate}
\item[$(H_1)$] for some $r > s > q > 0$, $f(u_t) v_t = t^q f(u) v$ and $g(u_t) v_t = t^r g(u) v$ for all $u, v \in W,\, t \ge 0$,
\item[$(H_2)$] if $u_j \wto u$ in $W$, then $f(u_j) \to f(u)$ and $g(u_j) \to g(u)$ in $W^\ast$, in particular, $f$ and $g$ map bounded sets into bounded sets.
\end{enumerate}

Solutions of equation \eqref{18} coincide with critical points of the $C^1$-functional
\[
\Phi_\lambda(u) = I_s(u) - \lambda\, J_s(u) - F(u) - G(u), \quad u \in W,
\]
where the potentials
\[
F(u) = \int_0^1 f(\tau u) u\, d\tau, \quad G(u) = \int_0^1 g(\tau u) u\, d\tau
\]
of $f$ and $g$, respectively, are even, bounded on bounded sets, and have the scaling properties
\begin{equation} \label{19}
F(u_t) = t^q\, F(u), \quad G(u_t) = t^r\, G(u) \quad \forall u \in W,\, t \ge 0
\end{equation}
(see \cite[Proposition 2.2]{MePe2}). It follows from $(H_2)$ that if $u_j \wto u$ in $W$, then $F(u_j) \to F(u)$ and $G(u_j) \to G(u)$ (see {\cite[Proposition 2.4]{MePe2}}). We assume that
\begin{enumerate}
\item[$(H_3)$] for any $\alpha, \beta, \gamma, \delta \in \R$, every solution of the equation $\alpha \As(u) = \beta \Bs(u) + \gamma f(u) + \delta g(u)$ satisfies the Poho\v{z}aev type identity
    \[
    s \alpha\, I_s(u) = s \beta\, J_s(u) + q \gamma\, F(u) + r \delta\, G(u),
    \]
\item[$(H_4)$] one of the following six cases holds:
    \begin{enumerate}
    \item[$(i)$] $F(u) > 0$ and $G(u) = 0$ for all $u \in W \setminus \set{0}$,
    \item[$(ii)$] $F(u) < 0$ and $G(u) = 0$ for all $u \in W \setminus \set{0}$,
    \item[$(iii)$] $F(u) = 0$ and $G(u) > 0$ for all $u \in W \setminus \set{0}$,
    \item[$(iv)$] $F(u) = 0$ and $G(u) < 0$ for all $u \in W \setminus \set{0}$,
    \item[$(v)$] $F(u) > 0$ and $G(u) < 0$ for all $u \in W \setminus \set{0}$,
    \item[$(vi)$] $F(u) < 0$ and $G(u) > 0$ for all $u \in W \setminus \set{0}$.
    \end{enumerate}
\end{enumerate}

\begin{remark}
If the mapping $[0,\infty) \to W,\, t \mapsto u_t$ is $C^1$ for each $u \in W$, then $(H_3)$ holds. Indeed, if $u$ is a solution of the equation $\alpha \As(u) = \beta \Bs(u) + \gamma f(u) + \delta g(u)$, then it is a critical point of the functional $\Psi(u) = \alpha\, I_s(u) - \beta\, J_s(u) - \gamma\, F(u) - \delta\, G(u)$, so
\[
\restr{\frac{d}{dt}}{t=1}\! \Psi(u_t) = \Psi'(u_1)\, \restr{\frac{d}{dt}}{t=1}\! u_t = 0
\]
since $u_1 = u$ by $(A_3)$. This reduces to the desired Poho\v{z}aev identity since $\Psi(u_t) = \alpha t^s I_s(u) - \beta t^s J_s(u) - \gamma t^q\, F(u) - \delta t^r\, G(u)$ by \eqref{3} and \eqref{19}.
\end{remark}

Given $c \in \R$, we are interested in finding pairs $(\lambda,u) \in \R \times W$ such that $\Phi_\lambda(u) = c$ and $\Phi_\lambda'(u) = 0$ (see Ramos Quoirin et al.\! \cite{MR4736027} and the references therein). For $u \ne 0$, the equation $\Phi_\lambda(u) = c$ has the unique solution for $\lambda$ given by
\begin{equation} \label{20}
\lambda_c(u) = \frac{I_s(u) - F(u) - G(u) - c}{J_s(u)}, \quad u \in W \setminus \set{0}
\end{equation}
(see \eqref{4}). The following proposition establishes the relationship between critical points of $\Phi_\lambda$ and those of $\lambda_c$.

\begin{proposition} \label{Proposition 7}
For $c \in \R$, $u \in W \setminus \set{0}$ is a critical point of $\Phi_\lambda$ with critical value $c$ if and only if it is a critical point of $\lambda_c$ with critical value $\lambda$.
\end{proposition}

\begin{proof}
We have
\begin{equation} \label{25}
\lambda_c(u) = \frac{\Phi_\lambda(u) - c}{J_s(u)} + \lambda
\end{equation}
and
\begin{equation} \label{9}
\lambda_c'(u) = \frac{\Phi_\lambda'(u) - (\lambda_c(u) - \lambda)\, \Bs(u)}{J_s(u)}.
\end{equation}
So $\Phi_\lambda(u) = c$ and $\Phi_\lambda'(u) = 0$ if and only if $\lambda_c(u) = \lambda$ and $\lambda_c'(u) = 0$.
\end{proof}

\subsection{Scaled Nehari manifold}

For $u \in W \setminus \set{0}$, set
\begin{equation} \label{10}
t_u = I_s(u)^{-1/s}, \quad \pi(u) = u_{t_u}
\end{equation}
(see \eqref{4}). Then $I_s(\pi(u)) = 1$ by \eqref{3} and hence $\pi(u) \in \M_s$. For $u \in \M_s$, $t_u = 1$ and hence $\pi(u) = u$ by $(A_3)$. So the mapping $W \setminus \set{0} \to \M_s,\, u \mapsto \pi(u)$ is a continuous projection on $\M_s$. We have
\begin{equation} \label{15}
\pi(u)_{t_u^{-1}} = u
\end{equation}
by $(A_1)$ and $(A_3)$, so
\begin{equation} \label{11}
W = \set{u_t : u \in \M_s,\, t \ge 0}.
\end{equation}

For $u \in \M_s$ and $t > 0$, set
\[
\varphi_{c,u}(t) = \lambda_c(u_t).
\]
We consider the associated scaled Nehari set
\[
\N_c = \set{u_t : u \in \M_s,\, t > 0,\, \varphi_{c,u}'(t) = 0}.
\]
By \eqref{20}, \eqref{3}, and \eqref{19},
\[
\varphi_{c,u}(t) = \frac{1 - t^{q-s}\, F(u) - t^{r-s}\, G(u) - c\, t^{-s}}{J_s(u)},
\]
so the equation $\varphi_{c,u}'(t) = 0$ is equivalent to
\begin{equation} \label{22}
(s - q)\, t^q\, F(u) - (r - s)\, t^r\, G(u) + cs = 0.
\end{equation}
This in turn is equivalent to
\[
(s - q)\, F(u_t) - (r - s)\, G(u_t) + cs = 0
\]
by \eqref{19}. In view of \eqref{11}, this implies that (see also Leite et al. \cite[Section 2.1]{LeQuSi} for the standard scaling)
\begin{equation} \label{23}
\N_c = \set{u \in W \setminus \set{0} : (s - q)\, F(u) - (r - s)\, G(u) + cs = 0}.
\end{equation}

First we note that all nontrivial solutions of equation \eqref{18} with energy $c$ belong to $\N_c$.

\begin{lemma} \label{Lemma 8}
If $u \in W \setminus \set{0}$ is a solution of the equation \eqref{18} with $\Phi_\lambda(u) = c$, then $u \in \N_c$.
\end{lemma}

\begin{proof}
We have
\[
s\, I_s(u) = s \lambda\, J_s(u) + q\, F(u) + r\, G(u)
\]
by $(H_3)$ and $I_s(u) - \lambda\, J_s(u) - F(u) - G(u) = c$, so $(s - q)\, F(u) - (r - s)\, G(u) + cs = 0$.
\end{proof}

We assume that
\begin{enumerate}
\item[$(H_5)$] $h(u) := (s - q)\, f(u) - (r - s)\, g(u) \ne 0$ for all $u \in \N_c$.
\end{enumerate}
Then $\N_c$ is a $C^1$-Finsler manifold. Since $F$ and $G$ are continuous and even, $\N_c$ is complete and symmetric. It is also weakly closed in view of \eqref{23}. We will refer to $\N_c$ as the scaled Nehari manifold.

Next we show that $\N_c$ is a natural constraint for $\lambda_c$. Let
\[
\Lambda_c = \restr{\lambda_c}{\N_c}.
\]
By \eqref{20} and \eqref{23},
\begin{equation} \label{24}
\Lambda_c(u) = \frac{(r - s)\, I_s(u) - (r - q)\, F(u) - cr}{(r - s)\, J_s(u)}, \quad u \in \N_c.
\end{equation}

\begin{lemma} \label{Lemma 6}
If $u \in \N_c$ is a critical point of $\Lambda_c$, then it is also a critical point of $\lambda_c$.
\end{lemma}

\begin{proof}
In view of \eqref{23},
\[
\lambda_c'(u) = \mu\, [(s - q)\, f(u) - (r - s)\, g(u)]
\]
for some Lagrange multiplier $\mu \in \R$. Combining this with \eqref{9} gives
\[
\As(u) = \lambda_c(u) \Bs(u) + [1 + \mu\, (s - q)\, J_s(u)]\, f(u) + [1 - \mu\, (r - s)\, J_s(u)]\, g(u).
\]
So $u$ satisfies
\[
s\, I_s(u) = s\, \lambda_c(u)\, J_s(u) + q\, [1 + \mu\, (s - q)\, J_s(u)]\, F(u) + r\, [1 - \mu\, (r - s)\, J_s(u)]\, G(u)
\]
by $(H_3)$. By \eqref{20} and \eqref{23}, this reduces to
\[
\mu\, J_s(u)\, [q\, (s - q)\, F(u) - r\, (r - s)\, G(u)] = 0.
\]
Since $r > s > q > 0$, $q\, (s - q)\, F(u) - r\, (r - s)\, G(u) \ne 0$ by $(H_4)$, and $J_s(u) \ne 0$ by \eqref{4}, so $\mu = 0$.
\end{proof}

Next we show that $\Lambda_c$ is bounded from below.

\begin{lemma}
$\Lambda_c$ is bounded from below on $\N_c$.
\end{lemma}

\begin{proof}
Suppose not. Since $r > s$, then it follows from \eqref{24} that there is a sequence $\seq{u_j}$ in $\N_c$ such that
\begin{equation} \label{28}
(r - s)\, I_s(u_j) < (r - q)\, F(u_j) + cr
\end{equation}
and
\begin{equation} \label{29}
J_s(u_j) \to 0.
\end{equation}
Set
\[
t_j = t_{u_j} = I_s(u_j)^{-1/s}, \quad \widetilde{u}_j = \pi(u_j) = (u_j)_{t_j}, \quad \widetilde{t}_j = t_j^{-1} = I_s(u_j)^{1/s}
\]
(see \eqref{10}). Then $\widetilde{u}_j \in \M_s$, $u_j = (\widetilde{u}_j)_{\widetilde{t}_j}$ by \eqref{15}, and \eqref{28} together with \eqref{19} gives
\[
(r - s)\, \widetilde{t}_j^{\, s} < (r - q)\, \widetilde{t}_j^{\, q}\, F(\widetilde{u}_j) + cr.
\]
Since $r > s > q > 0$ and $F(\widetilde{u}_j)$ is bounded since $\M_s$ is bounded and $F$ is bounded on bounded sets, it follows that $\widetilde{t}_j$ is bounded. So $\seq{u_j}$ is bounded by $(A_{10})$ and hence converges weakly to some $u \in W$ for a renamed subsequence. Then $J_s(u_j) \to J_s(u)$ and hence $J_s(u) = 0$ by \eqref{29}, so $u = 0$ by \eqref{4}. However, since $\N_c$ is weakly closed, $u \in \N_c \subset W \setminus \set{0}$, a contradiction.
\end{proof}

Our next lemma shows that $\Lambda_c$ is coercive, uniformly in $c$ on bounded sets.

\begin{lemma} \label{Lemma 9}
$\Lambda_c(u) \to \infty$ as $\norm{u} \to \infty$ in $\N_c$, uniformly in $c$ on bounded sets.
\end{lemma}

\begin{proof}
Let $\seq{c_j}$ be a bounded sequence and let $\seq{u_j}$ be a sequence in $\N_c$ such that $\norm{u_j} \to \infty$. Set
\[
t_j = t_{u_j} = I_s(u_j)^{-1/s}, \quad \widetilde{u}_j = \pi(u_j) = (u_j)_{t_j}, \quad \widetilde{t}_j = t_j^{-1} = I_s(u_j)^{1/s}
\]
(see \eqref{10}). Then $\widetilde{u}_j \in \M_s$ and
\begin{equation} \label{27}
u_j = (\widetilde{u}_j)_{\widetilde{t}_j}
\end{equation}
by \eqref{15}. Since $\M_s$ is a bounded manifold, $\seq{\widetilde{u}_j}$ is bounded and hence converges weakly to some $\widetilde{u} \in W$ for a renamed subsequence. Since $\norm{u_j} \to \infty$, $I_s(u_j) \to \infty$ by $(A_{10})$ and hence $\widetilde{t}_j \to \infty$. By \eqref{23}, \eqref{27}, and \eqref{19},
\begin{equation} \label{30}
(s - q)\, \widetilde{t}_j^{\, -(r-q)}\, F(\widetilde{u}_j) - (r - s)\, G(\widetilde{u}_j) + c_j\, s\, \widetilde{t}_j^{\, -r} = 0.
\end{equation}
Since $\widetilde{t}_j \to \infty$, $r > s > q > 0$, $F(\widetilde{u}_j) \to F(\widetilde{u})$, $G(\widetilde{u}_j) \to G(\widetilde{u})$, and $c_j$ is bounded, this implies that $G(\widetilde{u}) = 0$, so $\widetilde{u} = 0$ in cases $(iii)$--$(vi)$ of $(H_4)$. In cases $(i)$ and $(ii)$, \eqref{30} reduces to
\[
(s - q)\, F(\widetilde{u}_j) + c_j\, s\, \widetilde{t}_j^{\, -q} = 0,
\]
which implies that $F(\widetilde{u}) = 0$, so $\widetilde{u} = 0$ again. So $J_s(\widetilde{u}_j) \to J_s(0) = 0$. Since
\[
\Lambda_{c_j}(u_j) = \frac{(r - s)\, I_s(u_j) - (r - q)\, F(u_j) - c_j\, r}{(r - s)\, J_s(u_j)} = \frac{r - s - (r - q)\, \widetilde{t}_j^{\, -(s-q)}\, F(\widetilde{u}_j) - c_j\, r\, \widetilde{t}_j^{\, -s}}{(r - s)\, J_s(\widetilde{u}_j)}
\]
by \eqref{24}, \eqref{27}, \eqref{3}, and \eqref{19}, then $\Lambda_{c_j}(u_j) \to \infty$.
\end{proof}

Finally we show that $\Lambda_c$ satisfies the \PS{} condition.

\begin{lemma}
$\Lambda_c$ satisfies the {\em \PS{}} condition, i.e., every sequence $\seq{u_j}$ in $\N_c$ such that $\Lambda_c(u_j)$ is bounded and $\Lambda_c'(u_j) \to 0$ has a strongly convergent subsequence.
\end{lemma}

\begin{proof}
In view of \eqref{23},
\[
\lambda_c'(u_j) = \mu_j\, h(u_j) + \o(1)
\]
for some sequence of Lagrange multipliers $\seq{\mu_j}$ in $\R$. Combining this with \eqref{9} gives
\begin{equation} \label{26}
\As(u_j) = \Lambda_c(u_j) \Bs(u_j) + f(u_j) + g(u_j) + \gamma_j\, h(u_j) + \o(J_s(u_j)),
\end{equation}
where $\gamma_j = \mu_j\, J_s(u_j)$.

By Lemma \ref{Lemma 9}, $\seq{u_j}$ is bounded and hence converges weakly to some $u \in W$ for a renamed subsequence. Then $u \in \N_c$ since $\N_c$ is weakly closed, so $h(u) \ne 0$ by $(H_5)$. Since $h(u_j) \to h(u)$ in $W^\ast$ by $(H_2)$, then $h(u_j)$ is bounded away from zero in $W^\ast$. Since $\As$, $\Bs$, $f$, $g$, and $J_s$ are bounded on bounded sets and $\Lambda_c(u_j)$ is bounded by assumption, it now follows from \eqref{26} that $\gamma_j$ is bounded. So applying \eqref{26} to $u_j - u$ and noting that $\Bs(u_j)(u_j - u) \to 0$ by $(A_9)$ and $f(u_j)(u_j - u) \to 0$ and $g(u_j)(u_j - u) \to 0$ by $(H_2)$ (see Perera et al.\! \cite[Lemma 3.4]{MR2640827}), we have $\As(u_j)(u_j - u) \to 0$, so $u_j \to u$ for a further subsequence by $(A_7)$.
\end{proof}

\subsection{Minimax critical values}

Write $I^+ = (0,\infty)$ and $I^- = (- \infty,0)$. Then for each $u \in \M_s$, equation \eqref{22} has a unique positive solution $t = t_c(u)$ in the following six cases:
\begin{enumroman}
\item $F(u) > 0$ and $G(u) = 0$ for all $u \in W \setminus \set{0}$ and $c \in I^-$,
\item $F(u) < 0$ and $G(u) = 0$ for all $u \in W \setminus \set{0}$ and $c \in I^+$,
\item $F(u) = 0$ and $G(u) > 0$ for all $u \in W \setminus \set{0}$ and $c \in I^+$,
\item $F(u) = 0$ and $G(u) < 0$ for all $u \in W \setminus \set{0}$ and $c \in I^-$,
\item $F(u) > 0$ and $G(u) < 0$ for all $u \in W \setminus \set{0}$ and $c \in I^-$,
\item $F(u) < 0$ and $G(u) > 0$ for all $u \in W \setminus \set{0}$ and $c \in I^+$.
\end{enumroman}
We will simply write $I$ for $I^+$ or $I^-$ when there is no need to differentiate between these cases.

For each $c \in I$, we can define an increasing and unbounded sequence of critical values of $\Lambda_c$ as follows. Let $\F_c$ denote the class of symmetric subsets of $\N_c$. For $k \ge 1$, let $\F_{c,k} = \bgset{M \in \F_c : i(M) \ge k}$, set
\begin{equation} \label{50}
\lambda_{c,k} := \inf_{M \in \F_{c,k}}\, \sup_{u \in M}\, \Lambda_c(u),
\end{equation}
and set $\lambda_{0,k} = \lambda_k$ (see \eqref{5}). We have the following proposition (see Perera et al.\! \cite[Proposition 3.52]{MR2640827}).

\begin{proposition} \label{Proposition 3}
$\lambda_{c,k} \nearrow \infty$ is a sequence of critical values of $\Lambda_c$.
\end{proposition}

Note that
\[
\N_c = \set{u_{t_c(u)} : u \in \M_s}.
\]
Set
\[
\widetilde{\Lambda}_c(u) = \Lambda_c(u_{t_c(u)}), \quad u \in \M_s.
\]
The mapping
\[
\M_s \to \N_c, \quad u \mapsto u_{t_c(u)}
\]
is an odd homeomorphism with the inverse $\pi$ (see \eqref{10}), so it follows from the monotonicity of the index that
\begin{equation} \label{31}
\lambda_{c,k} = \inf_{M \in \F_k}\, \sup_{u \in M}\, \widetilde{\Lambda}_c(u)
\end{equation}
(see \eqref{5}).

\subsection{Analysis of the curves}

In this section we present a detailed analysis of the curves $C_k = \set{(\lambda_{c,k},c) : c \in I},\, k \ge 1$. The arguments here are adapted from Ramos Quoirin et al.\! \cite[Section 4]{MR4736027}.

By \eqref{20}, \eqref{3}, and \eqref{19},
\begin{equation} \label{37}
\widetilde{\Lambda}_c(u) = \frac{1 - t_c(u)^{-(s-q)}\, F(u) - t_c(u)^{r-s}\, G(u) - c\, t_c(u)^{-s}}{J_s(u)}, \quad (u,c) \in \M_s \times I.
\end{equation}
For fixed $u \in \M_s$, let
\[
H(c,t) = (s - q)\, t^q\, F(u) - (r - s)\, t^r\, G(u) + cs, \quad (c,t) \in I \times (0,\infty)
\]
and note that
\begin{equation} \label{32}
H(c,t_c(u)) = (s - q)\, t_c(u)^q\, F(u) - (r - s)\, t_c(u)^r\, G(u) + cs = 0.
\end{equation}
We have
\[
\frac{\partial H}{\partial t}(c,t_c(u)) = q\, (s - q)\, t_c(u)^{q-1}\, F(u) - r\, (r - s)\, t_c(u)^{r-1}\, G(u) \ne 0
\]
by $(H_4)$, so it follows from the implicit function theorem that the mapping $I \to (0,\infty),\, c \mapsto t_c(u)$ is $C^1$. Then the mapping $c \mapsto \widetilde{\Lambda}_c(u)$ is also $C^1$ and
\begin{equation} \label{33}
\frac{\partial \widetilde{\Lambda}_c(u)}{\partial c} = \frac{t_c(u)^{-s-1} H(c,t_c(u))\, \dfrac{\partial t_c(u)}{\partial c} - t_c(u)^{-s}}{J_s(u)} = - \frac{t_c(u)^{-s}}{J_s(u)}
\end{equation}
by \eqref{32}. First we note that the curves $C_k$ are nonincreasing.

\begin{lemma} \label{Lemma 1}
The following hold.
\begin{enumroman}
\item For each $u \in \M_s$, the mapping $I \to \R,\, c \mapsto \widetilde{\Lambda}_c(u)$ is decreasing.
\item For each $k \ge 1$, the mapping $I \to \R,\, c \mapsto \lambda_{c,k}$ is nonincreasing.
\end{enumroman}
\end{lemma}

\begin{proof}
$(i)$ holds since $\dfrac{\partial \widetilde{\Lambda}_c(u)}{\partial c} < 0$ by
\eqref{33} and \eqref{4}. $(ii)$ follows from $(i)$ and \eqref{31}.
\end{proof}

Lemma \ref{Lemma 1} implies that we can work in a suitable sublevel set $\widetilde{\Lambda}_c^T = \bgset{u \in \M_s : \widetilde{\Lambda}_c(u) \le T}$ in \eqref{31}.

\begin{lemma} \label{Lemma 2}
Let $[a,b] \subset I$, let $k \ge 1$, and let $T > \lambda_{a,k}$. Denote by $\F_{b,T}$ the class of symmetric subsets of $\widetilde{\Lambda}_b^T$ and let $\F_{b,T,k} = \bgset{M \in \F_{b,T} : i(M) \ge k}$. Then
\[
\inf_{M \in \F_{b,T,k}}\, \sup_{u \in M}\, \widetilde{\Lambda}_c(u) = \lambda_{c,k} \quad \forall c \in [a,b].
\]
\end{lemma}

\begin{proof}
Clearly, $\F_{b,T,k} \subset \F_k$. If $M \in \F_k \setminus \F_{b,T,k}$, then
\[
\sup_{u \in M}\, \widetilde{\Lambda}_c(u) \ge \sup_{u \in M}\, \widetilde{\Lambda}_b(u) > T >	\lambda_{a,k} \ge \lambda_{c,k}
\]
by Lemma \ref{Lemma 1}.
\end{proof}

We note that in view of \eqref{32}, $\widetilde{\Lambda}_c(u)$ can also be written as
\begin{equation} \label{40}
\widetilde{\Lambda}_c(u) = \frac{r - s - (r - q)\, t_c(u)^{-(s-q)}\, F(u) - cr\, t_c(u)^{-s}}{(r - s)\, J_s(u)}
\end{equation}
and
\begin{equation} \label{41}
\widetilde{\Lambda}_c(u) = \frac{s - q\, t_c(u)^{-(s-q)}\, F(u) - r\, t_c(u)^{r-s}\, G(u)}{s\, J_s(u)}.
\end{equation}
Next we prove the following estimates.

\begin{lemma} \label{Lemma 3}
Let $[a,b] \subset I$ and let $T \in \R$. Then $\exists C > 0$ such that the following hold.
\begin{enumroman}
\item $J_s(u) \ge C^{-1} \quad \forall u \in \widetilde{\Lambda}_b^T$
\item $C^{-1} \le t_c(u) \le C \quad \forall (u,c) \in \widetilde{\Lambda}_b^T \times [a,b]$
\item $- C \le \dfrac{\partial \widetilde{\Lambda}_c(u)}{\partial c} \le - C^{-1} \quad \forall (u,c) \in \widetilde{\Lambda}_b^T \times [a,b]$
\end{enumroman}
\end{lemma}

\begin{proof}
$(i)$ Suppose there is a sequence $\seq{u_j}$ in $\widetilde{\Lambda}_b^T$ such that $J_s(u_j) \to 0$. Since $\M_s$ is bounded, $\seq{u_j}$ is bounded and hence converges weakly to some $u \in W$ for a renamed subsequence. Then $J_s(u_j) \to J_s(u)$, so $J_s(u) = 0$. Then $u = 0$ by \eqref{4}, so $F(u_j) \to F(0) = 0$ and $G(u_j) \to G(0) = 0$. Since $b \ne 0$, then $t_b(u_j) \to \infty$ by \eqref{32}. Then $\widetilde{\Lambda}_b(u_j) \to \infty$ by \eqref{40}, contradicting $\widetilde{\Lambda}_b(u_j) \le T$.

$(ii)$ If the first inequality does not hold, then there are sequences $\seq{u_j}$ in $\widetilde{\Lambda}_b^T$ and $\seq{c_j}$ in $[a,b]$ such that $t_{c_j}(u_j) \to 0$. Since $F$ and $G$ are bounded on bounded sets, then $c_j \to 0$ by \eqref{32}, contradicting $0 \notin [a,b]$.

If the second inequality does not hold, then there are sequences $\seq{u_j}$ in $\widetilde{\Lambda}_b^T$ and $\seq{c_j}$ in $[a,b]$ such that $t_{c_j}(u_j) \to \infty$. A renamed subsequence of $\seq{u_j}$ converges weakly to some $u \in W$. Then $F(u_j) \to F(u)$ and $G(u_j) \to G(u)$, so \eqref{32} implies that $G(u) = 0$. So $u = 0$ in cases $(iii)$--$(vi)$ of $(H_4)$. In cases $(i)$ and $(ii)$, \eqref{32} reduces to
\[
(s - q)\, F(u_j) + c_j\, s\, t_{c_j}(u_j)^{-q} = 0,
\]
which implies that $F(u) = 0$. So $u = 0$ again. So $J_s(u_j) \to J_s(0) = 0$, contradicting $(i)$.

$(iii)$ This is immediate from \eqref{33}, $(i)$, and $(ii)$ since $J_s$ is bounded on bounded sets.
\end{proof}

We can now show that the curves $C_k$ are continuous and decreasing.

\begin{proposition} \label{Proposition 1}
For each $k \ge 1$, the mapping $I \to \R,\, c \mapsto \lambda_{c,k}$ is continuous and decreasing.
\end{proposition}

\begin{proof}
Let $[a,b] \subset I$ and let $T > \lambda_{a,k}$. Then
\begin{equation} \label{39}
\inf_{M \in \F_{b,T,k}}\, \sup_{u \in M}\, \widetilde{\Lambda}_c(u) = \lambda_{c,k} \quad \forall c \in [a,b]
\end{equation}
by Lemma \ref{Lemma 2}. For each $u \in \widetilde{\Lambda}_b^T$,
\[
\widetilde{\Lambda}_b(u) - \widetilde{\Lambda}_a(u) = \frac{\partial \widetilde{\Lambda}_c(u)}{\partial c}\, (b - a)
\]
for some $c \in (a,b)$ by the mean value theorem and hence
\[
- C\, (b - a) \le \widetilde{\Lambda}_b(u) - \widetilde{\Lambda}_a(u) \le - C^{-1}\, (b - a)
\]
by Lemma \ref{Lemma 3} $(iii)$. This together with \eqref{39} gives
\[
- C\, (b - a) \le \lambda_{b,k} - \lambda_{a,k} \le - C^{-1}\, (b - a),
\]
from which the desired conclusions follow.
\end{proof}

Next we establish the following limits for $\widetilde{\Lambda}_c$.

\begin{lemma} \label{Lemma 4}
Let $K \subset \M_s$ be a compact set.
\begin{enumroman}
\item If $F(u) > 0$ and $G(u) \le 0$ for all $u \in W \setminus \set{0}$, then
    \[
    \lim_{c \to 0^-}\, \sup_{u \in K}\, \widetilde{\Lambda}_c(u) = - \infty.
    \]
\item If $F(u) \le 0$ and $G(u) > 0$ for all $u \in W \setminus \set{0}$, then
    \[
    \lim_{c \to \infty}\, \sup_{u \in K}\, \widetilde{\Lambda}_c(u) = - \infty.
    \]
\item If $F(u) < 0$ and $G(u) \ge 0$ for all $u \in W \setminus \set{0}$, then
    \[
    \lim_{c \to 0^+}\, \inf_{u \in \M_s}\, \widetilde{\Lambda}_c(u) = \infty.
    \]
\item If $F(u) \ge 0$ and $G(u) < 0$ for all $u \in W \setminus \set{0}$, then
    \[
    \lim_{c \to - \infty}\, \inf_{u \in \M_s}\, \widetilde{\Lambda}_c(u) = \infty.
    \]
\end{enumroman}
\end{lemma}

\begin{proof}
$(i)$ Since $F > 0$ and $G \le 0$, it follows from \eqref{32} that $t_c(u)^q\, F(u) \to 0$ as $c \to 0^-$, uniformly on $\M_s$. Since $F$ is bounded away from zero on the compact set $K$, this implies that $t_c(u) \to 0$ as $c \to 0^-$, uniformly on $K$. The desired conclusion now follows from \eqref{41} since $G$ and $J$ are bounded on bounded sets.

$(ii)$ Since $F$ and $G$ are bounded on bounded sets, it follows from \eqref{32} that $t_c(u) \to \infty$ as $c \to \infty$, uniformly on $\M_s$. Since $G > 0$ is bounded away from zero on the compact set $K$ and $J$ is bounded on bounded sets, the desired conclusion now follows from \eqref{41}.

$(iii)$ Suppose not. Then there are sequences $\seq{u_j}$ in $\M_s$ and $c_j \to 0^+$ such that $\widetilde{\Lambda}_{c_j}(u_j) \le C$ for some constant $C > 0$. Since $F(u_j) < 0$ and $G(u_j) \ge 0$, it follows from \eqref{32} that $t_{c_j}(u_j)^q\, F(u_j) \to 0$ and $t_{c_j}(u_j)^r\, G(u_j) \to 0$. If $F(u_j)$ is bounded away from zero, this implies that $t_{c_j}(u_j) \to 0$. Since $G(u_j)$ and $J(u_j)$ are bounded, then it follows from \eqref{41} that $\widetilde{\Lambda}_{c_j}(u_j) \to \infty$, contradicting $\widetilde{\Lambda}_{c_j}(u_j) \le C$. So we may assume that $F(u_j) \to 0$ for a renamed subsequence. A further subsequence of $\seq{u_j}$ converges weakly to some $u \in W$. Then $F(u_j) \to F(u)$, so $F(u) = 0$ and hence $u = 0$. So $J_s(u_j) \to J_s(0) = 0$. Either $t_{c_j}(u_j)$ is bounded away from zero, or $t_{c_j}(u_j) \to 0$ for a subsequence. In the former case, $t_{c_j}(u_j)^{r-s}\, G(u_j) \to 0$ since $t_{c_j}(u_j)^r\, G(u_j) \to 0$. In the latter case, $t_{c_j}(u_j)^{r-s}\, G(u_j) \to 0$ since $G(u_j)$ is bounded. Since $F(u_j) < 0$, then
\[
\widetilde{\Lambda}_{c_j}(u_j) > \frac{s - r\, t_{c_j}(u_j)^{r-s}\, G(u_j)}{s\, J_s(u_j)} \to \infty,
\]
again contradicting $\widetilde{\Lambda}_{c_j}(u_j) \le C$. So the desired conclusion holds.

$(iv)$ Suppose not. Then there are sequences $\seq{u_j}$ in $\M_s$ and $c_j \to - \infty$ such that $\widetilde{\Lambda}_{c_j}(u_j) \le C$ for some constant $C > 0$. Since $F(u_j)$ and $G(u_j)$ are bounded, it follows from \eqref{32} that $t_{c_j}(u_j) \to \infty$. If $G(u_j) < 0$ is bounded away from zero, then it follows from \eqref{41} that $\widetilde{\Lambda}_{c_j}(u_j) \to \infty$ since $F(u_j)$ and $J(u_j)$ are bounded, contradicting $\widetilde{\Lambda}_{c_j}(u_j) \le C$. So we may assume that $G(u_j) \to 0$ for a renamed subsequence. A further subsequence of $\seq{u_j}$ converges weakly to some $u \in W$. Then $G(u_j) \to G(u)$, so $G(u) = 0$ and hence $u = 0$. So $J_s(u_j) \to J_s(0) = 0$. Since $G(u_j) < 0$ and $t_{c_j}(u_j)^{-(s-q)}\, F(u_j) \to 0$, then
\[
\widetilde{\Lambda}_{c_j}(u_j) > \frac{s - q\, t_{c_j}(u_j)^{-(s-q)}\, F(u_j)}{s\, J_s(u_j)} \to \infty,
\]
again contradicting $\widetilde{\Lambda}_{c_j}(u_j) \le C$. So the desired conclusion holds.
\end{proof}

Next we note that we can work in a suitable sublevel set $\widetilde{\Psi}^T = \bgset{u \in \M_s : \widetilde{\Psi}(u) \le T}$ in \eqref{5}.

\begin{lemma} \label{Lemma 5}
Let $k \ge 1$ and let $T > \lambda_k$. Denote by $\F_T$ the class of symmetric subsets of $\widetilde{\Psi}^T$ and let $\F_{T,k} = \bgset{M \in \F_T : i(M) \ge k}$. Then
\[
\inf_{M \in \F_{T,k}}\, \sup_{u \in M}\, \widetilde{\Psi}(u) = \lambda_k.
\]
\end{lemma}

\begin{proof}
Clearly, $\F_{T,k} \subset \F_k$. If $M \in \F_k \setminus \F_{T,k}$, then
\[
\sup_{u \in M}\, \widetilde{\Psi}(u) > T > \lambda_k. \QED
\]
\end{proof}

Finally we establish the following limits for $\lambda_{c,k}$.

\begin{proposition} \label{Proposition 2}
For each $k \ge 1$, the following hold.
\begin{enumroman}
\item If $F(u) > 0$ and $G(u) \le 0$ for all $u \in W \setminus \set{0}$, then $\ds{\lim_{c \to 0^-}}\, \lambda_{c,k} = - \infty$. If, in addition, $G = 0$, then
    $\ds{\lim_{c \to - \infty}}\, \lambda_{c,k} = \lambda_k$.
\item If $F(u) \le 0$ and $G(u) > 0$ for all $u \in W \setminus \set{0}$, then $\ds{\lim_{c \to \infty}}\, \lambda_{c,k} = - \infty$. If, in addition, $F = 0$, then
    $\ds{\lim_{c \to 0^+}}\, \lambda_{c,k} = \lambda_k$.
\item If $F(u) < 0$ and $G(u) \ge 0$ for all $u \in W \setminus \set{0}$, then $\ds{\lim_{c \to 0^+}}\, \lambda_{c,k} = \infty$. If, in addition, $G = 0$, then
    $\ds{\lim_{c \to \infty}}\, \lambda_{c,k} = \lambda_k$.
\item If $F(u) \ge 0$ and $G(u) < 0$ for all $u \in W \setminus \set{0}$, then $\ds{\lim_{c \to - \infty}}\, \lambda_{c,k} = \infty$. If, in addition, $F = 0$, then
    $\ds{\lim_{c \to 0^-}}\, \lambda_{c,k} = \lambda_k$.
\end{enumroman}
\end{proposition}

\begin{proof}
$(i)$ Let $K \in \F_k$ be a compact set. Then $\lambda_{c,k} \le \ds{\sup_{u \in K}}\, \widetilde{\Lambda}_c(u)$, so the first limit follows from Lemma \ref{Lemma 4} $(i)$. If, in addition, $G = 0$, then \eqref{41} together with \eqref{32} gives
\begin{equation} \label{42}
\widetilde{\Lambda}_c(u) = \widetilde{\Psi}(u) - \frac{q}{s}\, \left(\frac{s - q}{|c|\, s}\right)^{(s-q)/q} \frac{F(u)^{s/q}}{J_s(u)} \le \widetilde{\Psi}(u) \quad \forall u \in \M_s,
\end{equation}
so $\lambda_{c,k} \le \lambda_k$ for all $c \in I^-$ by \eqref{31} and \eqref{5}. Let $b \in I^-$ and let $T > \lambda_k$. Then $T > \lambda_{c,k}$, so
\begin{equation} \label{43}
\inf_{M \in \F_{b,T,k}}\, \sup_{u \in M}\, \widetilde{\Lambda}_c(u) = \lambda_{c,k} \quad \forall c \le b
\end{equation}
by Lemma \ref{Lemma 2}. The equality in \eqref{42} combined with Lemma \ref{Lemma 3} $(i)$ and the fact that $F$ is bounded on bounded sets gives
\[
\widetilde{\Lambda}_c(u) \ge \widetilde{\Psi}(u) - \frac{C}{|c|^{(s-q)/q}} \quad \forall u \in \widetilde{\Lambda}_b^T
\]
for some constant $C > 0$. This together with \eqref{43} gives
\[
\lambda_{c,k} \ge \inf_{M \in \F_{b,T,k}}\, \sup_{u \in M}\, \widetilde{\Psi}(u) - \frac{C}{|c|^{(s-q)/q}} \ge \lambda_k - \frac{C}{|c|^{(s-q)/q}} \quad \forall c \le b
\]
since $\F_{b,T,k} \subset \F_k$, and the second limit follows from this since $\lambda_{c,k} \le \lambda_k$.

$(ii)$ Let $K \in \F_k$ be a compact set. Then $\lambda_{c,k} \le \ds{\sup_{u \in K}}\, \widetilde{\Lambda}_c(u)$, so the first limit follows from Lemma \ref{Lemma 4} $(ii)$. If, in addition, $F = 0$, then \eqref{41} together with \eqref{32} gives
\begin{equation} \label{44}
\widetilde{\Lambda}_c(u) = \widetilde{\Psi}(u) - \frac{r}{s}\, \left(\frac{cs}{r - s}\right)^{(r-s)/r} \frac{G(u)^{s/r}}{J_s(u)} \le \widetilde{\Psi}(u) \quad \forall u \in \M_s,
\end{equation}
so $\lambda_{c,k} \le \lambda_k$ for all $c \in I^+$ by \eqref{31} and \eqref{5}. Let $b \in I^+$ and let $T > \lambda_k$. Then $T > \lambda_{c,k}$, so
\begin{equation} \label{45}
\inf_{M \in \F_{b,T,k}}\, \sup_{u \in M}\, \widetilde{\Lambda}_c(u) = \lambda_{c,k} \quad \forall c \in (0,b]
\end{equation}
by Lemma \ref{Lemma 2}. The equality in \eqref{44} combined with Lemma \ref{Lemma 3} $(i)$ and the fact that $G$ is bounded on bounded sets gives
\[
\widetilde{\Lambda}_c(u) \ge \widetilde{\Psi}(u) - C\, |c|^{(r-s)/r} \quad \forall u \in \widetilde{\Lambda}_b^T
\]
for some constant $C > 0$. This together with \eqref{45} gives
\[
\lambda_{c,k} \ge \inf_{M \in \F_{b,T,k}}\, \sup_{u \in M}\, \widetilde{\Psi}(u) - C\, |c|^{(r-s)/r} \ge \lambda_k - C\, |c|^{(r-s)/r} \quad \forall c \in (0,b]
\]
since $\F_{b,T,k} \subset \F_k$, and the second limit follows from this since $\lambda_{c,k} \le \lambda_k$.

$(iii)$ Since $\lambda_{c,k} \ge \ds{\inf_{u \in \M_s}}\, \widetilde{\Lambda}_c(u)$, the first limit follows from Lemma \ref{Lemma 4} $(iii)$. If, in addition, $G = 0$, then \eqref{41} together with \eqref{32} gives
\begin{equation} \label{46}
\widetilde{\Lambda}_c(u) = \widetilde{\Psi}(u) \left[1 + \frac{q}{s}\, \left(\frac{s - q}{cs}\right)^{(s-q)/q} |F(u)|^{s/q}\right] \ge \widetilde{\Psi}(u) \quad \forall u \in \M_s,
\end{equation}
so $\lambda_{c,k} \ge \lambda_k$ for all $c \in I^+$ by \eqref{31} and \eqref{5}. Let $T > \lambda_k$. Then
\begin{equation} \label{47}
\inf_{M \in \F_{T,k}}\, \sup_{u \in M}\, \widetilde{\Psi}(u) = \lambda_k
\end{equation}
by Lemma \ref{Lemma 5}. The equality in \eqref{46} combined with the fact that $F$ is bounded on bounded sets gives
\[
\widetilde{\Lambda}_c(u) \le \widetilde{\Psi}(u) + \frac{C}{|c|^{(s-q)/q}} \quad \forall u \in \widetilde{\Psi}^T
\]
for some constant $C > 0$. Since $\F_k \supset \F_{T,k}$, this together with \eqref{47} gives
\[
\lambda_{c,k} \le \inf_{M \in \F_{T,k}}\, \sup_{u \in M}\, \widetilde{\Lambda}_c(u) \le \lambda_k + \frac{C}{|c|^{(s-q)/q}} \quad \forall c \in I^+,
\]
and the second limit follows from this since $\lambda_{c,k} \ge \lambda_k$.

$(iv)$ Since $\lambda_{c,k} \ge \ds{\inf_{u \in \M_s}}\, \widetilde{\Lambda}_c(u)$, the first limit follows from Lemma \ref{Lemma 4} $(iv)$. If, in addition, $F = 0$, then \eqref{41} together with \eqref{32} gives
\begin{equation} \label{48}
\widetilde{\Lambda}_c(u) = \widetilde{\Psi}(u) \left[1 + \frac{r}{s}\, \left(\frac{|c|\, s}{r - s}\right)^{(r-s)/r} |G(u)|^{s/r}\right] \ge \widetilde{\Psi}(u) \quad \forall u \in \M_s,
\end{equation}
so $\lambda_{c,k} \ge \lambda_k$ for all $c \in I^-$ by \eqref{31} and \eqref{5}. Let $T > \lambda_k$. Then
\begin{equation} \label{49}
\inf_{M \in \F_{T,k}}\, \sup_{u \in M}\, \widetilde{\Psi}(u) = \lambda_k
\end{equation}
by Lemma \ref{Lemma 5}. The equality in \eqref{48} combined with the fact that $G$ is bounded on bounded sets gives
\[
\widetilde{\Lambda}_c(u) \le \widetilde{\Psi}(u) + C\, |c|^{(r-s)/r} \quad \forall u \in \widetilde{\Psi}^T
\]
for some constant $C > 0$. Since $\F_k \supset \F_{T,k}$, this together with \eqref{49} gives
\[
\lambda_{c,k} \le \inf_{M \in \F_{T,k}}\, \sup_{u \in M}\, \widetilde{\Lambda}_c(u) \le \lambda_k + C\, |c|^{(r-s)/r} \quad \forall c \in I^-,
\]
and the second limit follows from this since $\lambda_{c,k} \ge \lambda_k$.
\end{proof}

\subsection{Main results}

In this section we prove our main results for the equation \eqref{18}. The first and the second parts of each of the following theorems are immediate from Proposition \ref{Proposition 3}, Lemma \ref{Lemma 6}, Proposition \ref{Proposition 7}, Proposition \ref{Proposition 1}, and Proposition \ref{Proposition 2}, so we will only give the proof of the third part. These theorems are illustrated in the figures below. For each point of intersection of the line $L_\lambda = \set{(\lambda,c) : c \in I}$ with one of the curves $C_k = \set{(\lambda_{c,k},c) : c \in I},\, k \ge 1$, $\Phi_{\lambda_{c,k}}$ has a nontrivial critical point $u_{c,k}$ with critical value $c$. No nontrivial critical point exists below the first curve in each figure.

\begin{theorem} \label{Theorem 1}
Assume $(A_1)$--$(A_{11})$, $(H_1)$--$(H_3)$, $(H_4)(i)$, and $(H_5)$.
\begin{enumroman}
\item For each $c \in I^-$, $\lambda_{c,k} \nearrow \infty$ is a sequence of critical values of $\lambda_c$. Moreover, $\Phi_{\lambda_{c,k}}$ has a nontrivial critical point $u_{c,k}$ with $\Phi_{\lambda_{c,k}}(u_{c,k}) = c$.
\item For each $k \ge 1$, the mapping $I^- \to \R,\, c \mapsto \lambda_{c,k}$ is continuous, decreasing, and satisfies
    \[
    \lim_{c \to 0^-}\, \lambda_{c,k} = - \infty, \qquad \lim_{c \to - \infty}\, \lambda_{c,k} = \lambda_k.
    \]
\item For each $\lambda \in \R$, $\Phi_\lambda$ has a sequence of nontrivial critical points $\seq{v_{\lambda,k}}$ such that $\Phi_\lambda(v_{\lambda,k}) \nearrow 0$ and $\norm{v_{\lambda,k}} \to 0$. Moreover, $\Phi_\lambda$ has no nontrivial critical point with critical value in $\R \setminus I^-$.
\end{enumroman}
\end{theorem}

\begin{proof}
By $(ii)$, for all sufficiently large $k$, $L_\lambda \cap C_k \ne \emptyset$, i.e., there is a $c_k \in I^-$ such that $\lambda_{c_k,k} = \lambda$. So $\Lambda_{c_k}$ has a critical point $v_{\lambda,k}$ with critical value $\lambda$ by Proposition \ref{Proposition 3}. Then $v_{\lambda,k}$ is also a critical point of $\lambda_{c_k}$ with critical value $\lambda$ by Lemma \ref{Lemma 6} and hence it is a nontrivial critical point of $\Phi_\lambda$ with critical value $c_k$ by Proposition \ref{Proposition 7}. For any $c_0 \in I^-$, for all sufficiently larger $k$,
\[
\lambda_{c_0,k} > \lambda = \lambda_{c_k,k} = \lambda_{c_{k+1},{k+1}} \ge \lambda_{c_{k+1},k}
\]
since $\lambda_{c,k} \nearrow \infty$ for all $c \in I^-$. This implies that $c_{k+1} \ge c_k > c_0$ since $\lambda_{c,k}$ is decreasing in $c$, so $c_k \nearrow 0$. Since $\Lambda_{c_k}(v_{\lambda,k}) = \lambda$, then $\seq{v_{\lambda,k}}$ is bounded by Lemma \ref{Lemma 9} and hence converges weakly to some $v_\lambda \in W$ for a renamed subsequence. Then $F(v_{\lambda,k}) \to F(v_\lambda)$. On the other hand, since $v_{\lambda,k} \in \N_{c_k}$, $(s - q)\, F(v_{\lambda,k}) + c_k\, s = 0$ by \eqref{23}, so $F(v_{\lambda,k}) \to 0$. So $F(v_\lambda) = 0$ and hence $v_\lambda = 0$. Now applying $\Phi_\lambda'(v_{\lambda,k}) = 0$ to $v_{\lambda,k}$ and noting that $\Bs(v_{\lambda,k}) v_{\lambda,k} \to 0$ by $(A_9)$ and $f(v_{\lambda,k}) v_{\lambda,k} \to 0$ by $(H_2)$ (see Perera et al.\! \cite[Lemma 3.4]{MR2640827}), we have $\As(v_{\lambda,k}) v_{\lambda,k} \to 0$, so $v_{\lambda,k} \to 0$ for a further subsequence by $(A_7)$.

If $c \ge 0$, then the equation $(s - q)\, F(u) + cs = 0$ has no nontrivial solution since $F(u) > 0$ for all $u \in W \setminus \set{0}$. So $\N_c = \emptyset$ by \eqref{23} and hence $\Phi_\lambda$ has no nontrivial critical point with critical value $c$ by Lemma \ref{Lemma 8}.
\end{proof}

\begin{figure}[h!]
	\centering
	\begin{tikzpicture}[>=latex]
		%x axis
		\draw[->] (-3,0) -- (4,0) node[below] {\scalebox{0.8}{$\lambda$}};
		\foreach \x in {}
		\draw[shift={(\x,0)}] (0pt,2pt) -- (0pt,-2pt) node[below] {\footnotesize $\x$};
		%y axis
		\draw[->] (0,-2.5) -- (0,1) node[left] {\scalebox{0.8}{$c$}};
		\foreach \y in {}
		\draw[shift={(0,\y)}] (2pt,0pt) -- (-2pt,0pt) nde[left] {\footnotesize $\y$};
		%\node[below left] at (0,0) {\footnotesize $0$};
		%	\draw[red,thick] (0,-2) .. controls (0,-1.5) and (0,0) .. (4,1.5);
		\draw[blue,thick] (-2.5,-1) .. controls (0,-1.2) and (1,-1.2) .. (1.45,-2.5);
		%\draw[blue,thick] (2,0) .. controls (3,-1.4) .. (3,-1.4);
		%\draw[blue,thick,dashed] (0,1) .. controls (0,1) and (1,1) .. (2,0);
		%\draw[blue,thick] (-2.5,-.7) .. controls (0,-.9) and (1.5,-.9) .. (1.9,-2.5);
		%\draw[blue,thick] (3,0) .. controls (4,-1.4) .. (4,-1.4);
		\draw [thick] (1.5,-.1) node[above]{\scalebox{0.8}{$\lambda_{1}$}} -- (1.5,0.05); 
		\draw [dashed] (1.5,0) -- (1.5,-2.5);
		\draw [thick] (2.3,-.1) node[above]{\scalebox{0.8}{$\lambda_{2}$}} -- (2.3,0.05); 
		\draw [dashed] (2.3,0) -- (2.3,-2.5);
		%	\draw [thick] (2.5,-.1) node[below]{\scalebox{0.8}{$\mu_{3}$}} -- (2.5,0.05); 
		%				\draw [dashed] (2.5,0) -- (2.5,-2.5);
		\draw [thick] (3.3,-.1) node[above]{\scalebox{0.8}{$\lambda_{k}$}} -- (3.3,0.05); 
		\draw [dashed] (3.3,0) -- (3.3,-2.5);
		
		\draw [thick] (2.4,-2) node[above]{$\cdots$}; 
		\draw[blue,thick] (-2.5,-.5) .. controls (0,-.6) and (1.7,-.6) .. (2.25,-2.5);
		\draw[blue,thick] (-2.5,-.2) .. controls (0,-.3) and (2.5,-.3) .. (3.25,-2.5);
		
		%\draw[blue,thick] (4,0) .. controls (5,-1.4) .. (5,-1.4);
		%\draw [thick] (3.2,-1.3) node[above]{$\vdots$} -- (3.2,-1.3); 
		%\draw [thick] (-.1,.4) node[right]{$\vdots$} ;%-- (-.1,2,2);
		%\draw (.5,.35) node{$\bullet$};
		\draw [thick] (.75,0) -- (.75,-2.5);
		
		\draw (.75,-1.6) node{\scalebox{0.8}{$\bullet$}};
		\draw (.75,-.9) node{\scalebox{0.8}{$\bullet$}};
		\draw (.75,-.5) node{\scalebox{0.8}{$\bullet$}};
		%\draw (1.8,1.9) node{\scalebox{0.8}{$\bullet$}};
		%\draw (1.8,-.18) node{\scalebox{0.8}{$\bullet$}};
		%\draw (1.8,-.37) node{\scalebox{0.8}{$\bullet$}};
		%\draw (1.8,-.82) node{\scalebox{0.8}{$\bullet$}};
		\draw  (-1,-1.5) node[below]{\scalebox{1.5}{$\nexists$}}  ; 	
		\draw  (-2.8,-1.4) node[above]{\scalebox{0.8}{$\lambda_{c,1}$}} ; 	
		\draw  (-2.8,-.9) node[above]{\scalebox{0.8}{$\lambda_{c,2}$}} ; 
		%	\draw  (-2.8,0) node[above]{\scalebox{0.8}{$\mu_{3,c}^-$}} ; 
		
		\draw  (-2.8,-.5) node[above]{\scalebox{0.8}{$\lambda_{c,k}$}} ; 
		
		%\draw [thick] (-2.8,1.1) node[above]{\scalebox{0.8}{$\mu_{1,c}$}} -- (-2.8,1.2); 	
		%	\draw[blue,thick,dashed] (3,1.3) .. controls (3,1.3) and (4,1.5) .. (4,1.5);
		%\draw [thick] (1,-.1) node[below]{$\mu_1$} -- (1,0.1); 
		%\draw [thick] (2,-.05) node[above]{\scalebox{0.8}{$\mu_{1,0}$}} -- (2,0.1); 
		%\draw [thick] (3,-.05) node[below]{\scalebox{0.8}{$\mu_2$}} -- (3,0.1); 
		%\draw [thick] (4,-.05) node[below]{\scalebox{0.8}{$\mu_n$}} -- (4,.1); 
		%\draw [thick] (1.8,0) node[above left]{\scalebox{0.8}{\tiny$\overline{\mu}$}} -- (1.8,0); 
		%\draw [thick] (2.5,-0.05) node[above]{\scalebox{0.8}{$\mu^*$}} -- (2.5,0.1); 
		%	\draw [thick] (-.1,-2) node[left]{$-\infty$} -- (.1,-2); 
		%	\node[] at (1,1.5) { {\color{blue}$\Phi_\mu(w_\mu)$}};
		%	\node[] at (3,0.5) { {\color{red}$\Phi_\mu(u_\mu)$}};
		%\draw  (.8,0.1) node[above]{\scalebox{0.8}{$\Phi_\mu(u_\mu)$}} ; 
		%\draw  (2.5,-2) node[above]{\scalebox{0.8}{$\Phi_\mu(v_\mu)$}} ; 
		\draw  (.75,0) node[above]{\scalebox{0.8}{$L_\lambda$}} ;
	\end{tikzpicture}
	\caption{Energy curves for the case $F>0,\, G=0$ (see Theorem \ref{Theorem 1}).} \label{fig1}
\end{figure}
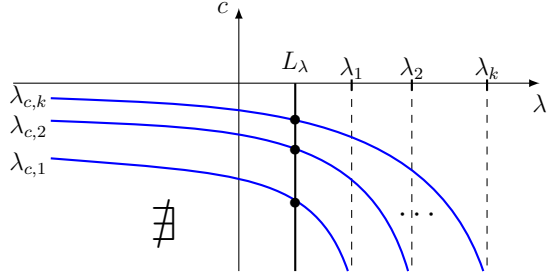

\begin{theorem} \label{Theorem 2}
Assume $(A_1)$--$(A_{11})$, $(H_1)$--$(H_3)$, $(H_4)(ii)$, and $(H_5)$.
\begin{enumroman}
\item For each $c \in I^+$, $\lambda_{c,k} \nearrow \infty$ is a sequence of critical values of $\lambda_c$. Moreover, $\Phi_{\lambda_{c,k}}$ has a nontrivial critical point $u_{c,k}$ with $\Phi_{\lambda_{c,k}}(u_{c,k}) = c$.
\item For each $k \ge 1$, the mapping $I^+ \to \R,\, c \mapsto \lambda_{c,k}$ is continuous, decreasing, and satisfies
    \[
    \lim_{c \to 0^+}\, \lambda_{c,k} = \infty, \qquad \lim_{c \to \infty}\, \lambda_{c,k} = \lambda_k.
    \]
\item For $\lambda \le \lambda_1$, $\Phi_\lambda$ has no nontrivial critical point. For each $\lambda > \lambda_k$, $\Phi_\lambda$ has at least $k$ pairs of nontrivial critical points. For all $\lambda \in \R$, $\Phi_\lambda$ has no nontrivial critical point with critical value in $\R \setminus I^+$.
\end{enumroman}
\end{theorem}

\begin{proof}
For $\lambda \le \lambda_1$, for all $c \in I^+$,
\[
\lambda < \lambda_{c,1} = \min_{u \in \N_c}\, \Lambda_c(u) = \min_{u \in \N_c}\, \lambda_c(u)
\]
by $(ii)$ and \eqref{50}. So $\lambda_c$ has no critical point with critical value $\lambda$ and hence $\Phi_\lambda$ has no nontrivial critical point with critical value $c$ by Proposition \ref{Proposition 7}.

For each $\lambda > \lambda_k$, by $(ii)$, for $j = 1,\dots,k$, $L_\lambda \cap C_j \ne \emptyset$, i.e., there is a $c_j \in I^+$ such that $\lambda_{c_j,j} = \lambda$. Since $\lambda_{c_j,j}$ is a critical value of $\lambda_{c_j}$ by Proposition \ref{Proposition 3} and Lemma \ref{Lemma 6}, then $\Phi_\lambda$ has a nontrivial critical point with critical value $c_j$ by Proposition \ref{Proposition 7}.

If $c \le 0$, then the equation $(s - q)\, F(u) + cs = 0$ has no nontrivial solution since $F(u) < 0$ for all $u \in W \setminus \set{0}$. So $\N_c = \emptyset$ by \eqref{23} and hence $\Phi_\lambda$ has no nontrivial critical point with critical value $c$ by Lemma \ref{Lemma 8}.
\end{proof}

\begin{figure}[h!]
	\centering
	\begin{tikzpicture}[>=latex]
		%x axis
		\draw[->] (-1,0) -- (6,0) node[below] {\scalebox{0.8}{$\lambda$}};
		\foreach \x in {}
		\draw[shift={(\x,0)}] (0pt,2pt) -- (0pt,-2pt) node[below] {\footnotesize $\x$};
		%y axis
		\draw[->] (0,-1) -- (0,2.5) node[left] {\scalebox{0.8}{$c$}};
		\foreach \y in {}
		\draw[shift={(0,\y)}] (2pt,0pt) -- (-2pt,0pt) nde[left] {\footnotesize $\y$};
		%\node[below left] at (0,0) {\footnotesize $0$};
		%	\draw[red,thick] (0,-2) .. controls (0,-1.5) and (0,0) .. (4,1.5);
		\draw[blue,thick] (1.55,2.5) .. controls (2,.2) and (3,0.2) .. (6,.2);
		%\draw[blue,thick] (2,0) .. controls (3,-1.4) .. (3,-1.4);
		%\draw[blue,thick,dashed] (0,1) .. controls (0,1) and (1,1) .. (2,0);
		%\draw[blue,thick] (-2.5,-.7) .. controls (0,-.9) and (1.5,-.9) .. (1.9,-2.5);
		%\draw[blue,thick] (3,0) .. controls (4,-1.4) .. (4,-1.4);
		\draw [thick] (1.5,-.1) node[below]{\scalebox{0.8}{$\lambda_{1}$}} -- (1.5,0.05); 
		\draw [dashed] (1.5,0) -- (1.5,2.5);
		\draw [thick] (2.3,-.1) node[below]{\scalebox{0.8}{$\lambda_{2}$}} -- (2.3,0.05); 
		\draw [dashed] (2.3,0) -- (2.3,2.5);
		%	\draw [thick] (2.5,-.1) node[below]{\scalebox{0.8}{$\mu_{3}$}} -- (2.5,0.05); 
		%				\draw [dashed] (2.5,0) -- (2.5,-2.5);
		\draw [thick] (3.3,-.1) node[below]{\scalebox{0.8}{$\lambda_{k}$}} -- (3.3,0.05); 
		\draw [dashed] (3.3,0) -- (3.3,2.5);
		
		\draw [thick] (4,.5) node[above]{$\cdots$}; 
		\draw[blue,thick] (2.35,2.5) .. controls (2.8,.4) and (3.8,0.4) .. (6,.4);
		\draw[blue,thick] (3.35,2.5) .. controls (3.8,.6) and (4.8,0.6) .. (6,.6);
		
		%\draw[blue,thick] (4,0) .. controls (5,-1.4) .. (5,-1.4);
		%\draw [thick] (3.2,-1.3) node[above]{$\vdots$} -- (3.2,-1.3); 
		%\draw [thick] (-.1,.4) node[right]{$\vdots$} ;%-- (-.1,2,2);
		%\draw (.5,.35) node{$\bullet$};
		\draw [thick] (5,0) -- (5,2.5);
		\draw (5,.64) node{\scalebox{0.8}{$\bullet$}};
		\draw (5,.4) node{\scalebox{0.8}{$\bullet$}};
		\draw (5,.2) node{\scalebox{0.8}{$\bullet$}};
		%\draw (1.8,1.9) node{\scalebox{0.8}{$\bullet$}};
		%\draw (1.8,-.18) node{\scalebox{0.8}{$\bullet$}};
		%\draw (1.8,-.37) node{\scalebox{0.8}{$\bullet$}};
		%\draw (1.8,-.82) node{\scalebox{0.8}{$\bullet$}};
		\draw  (1,1) node[below]{\scalebox{1.5}{$\nexists$}}  ; 	
		\draw  (1.5,2.5) node[above]{\scalebox{0.8}{$\lambda_{c,1}$}} ; 	
		\draw  (2.3,2.5) node[above]{\scalebox{0.8}{$\lambda_{c,2}$}} ; 
		%	\draw  (-2.8,0) node[above]{\scalebox{0.8}{$\mu_{3,c}^-$}} ; 
		
		\draw  (3.3,2.5) node[above]{\scalebox{0.8}{$\lambda_{c,k}$}} ; 
		
		%\draw [thick] (-2.8,1.1) node[above]{\scalebox{0.8}{$\mu_{1,c}$}} -- (-2.8,1.2); 	
		%	\draw[blue,thick,dashed] (3,1.3) .. controls (3,1.3) and (4,1.5) .. (4,1.5);
		%\draw [thick] (1,-.1) node[below]{$\mu_1$} -- (1,0.1); 
		%\draw [thick] (2,-.05) node[above]{\scalebox{0.8}{$\mu_{1,0}$}} -- (2,0.1); 
		%\draw [thick] (3,-.05) node[below]{\scalebox{0.8}{$\mu_2$}} -- (3,0.1); 
		%\draw [thick] (4,-.05) node[below]{\scalebox{0.8}{$\mu_n$}} -- (4,.1); 
		%\draw [thick] (1.8,0) node[above left]{\scalebox{0.8}{\tiny$\overline{\mu}$}} -- (1.8,0); 
		%\draw [thick] (2.5,-0.05) node[above]{\scalebox{0.8}{$\mu^*$}} -- (2.5,0.1); 
		%	\draw [thick] (-.1,-2) node[left]{$-\infty$} -- (.1,-2); 
		%	\node[] at (1,1.5) { {\color{blue}$\Phi_\mu(w_\mu)$}};
		%	\node[] at (3,0.5) { {\color{red}$\Phi_\mu(u_\mu)$}};
		%\draw  (.8,0.1) node[above]{\scalebox{0.8}{$\Phi_\mu(u_\mu)$}} ; 
		%\draw  (2.5,-2) node[above]{\scalebox{0.8}{$\Phi_\mu(v_\mu)$}} ; 
		\draw  (5,0) node[below]{\scalebox{0.8}{$L_\lambda$}} ;
	\end{tikzpicture}
	\caption{Energy curves for the case $F<0,\, G=0$ (see Theorem \ref{Theorem 2}).} \label{fig2}
\end{figure}
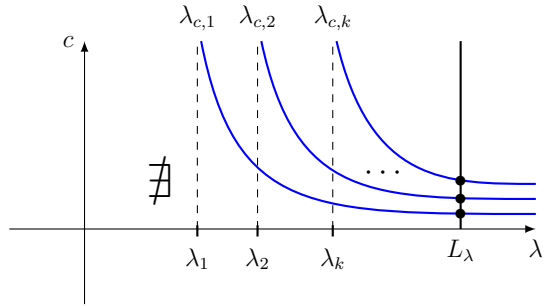

\begin{theorem} \label{Theorem 3}
Assume $(A_1)$--$(A_{11})$, $(H_1)$--$(H_3)$, $(H_4)(iii)$, and $(H_5)$.
\begin{enumroman}
\item For each $c \in I^+$, $\lambda_{c,k} \nearrow \infty$ is a sequence of critical values of $\lambda_c$. Moreover, $\Phi_{\lambda_{c,k}}$ has a nontrivial critical point $u_{c,k}$ with $\Phi_{\lambda_{c,k}}(u_{c,k}) = c$.
\item For each $k \ge 1$, the mapping $I^+ \to \R,\, c \mapsto \lambda_{c,k}$ is continuous, decreasing, and satisfies
    \[
    \lim_{c \to 0^+}\, \lambda_{c,k} = \lambda_k, \qquad \lim_{c \to \infty}\, \lambda_{c,k} = - \infty.
    \]
\item For each $\lambda \in \R$, $\Phi_\lambda$ has a sequence of nontrivial critical points $\seq{v_{\lambda,k}}$ such that $\Phi_\lambda(v_{\lambda,k}) \nearrow \infty$ and $\norm{v_{\lambda,k}} \to \infty$. Moreover, $\Phi_\lambda$ has no nontrivial critical point with critical value in $\R \setminus I^+$.
\end{enumroman}
\end{theorem}

\begin{proof}
By $(ii)$, for all sufficiently large $k$, $L_\lambda \cap C_k \ne \emptyset$, i.e., there is a $c_k \in I^+$ such that $\lambda_{c_k,k} = \lambda$. Since $\lambda_{c_k,k}$ is a critical value of $\lambda_{c_k}$ by Proposition \ref{Proposition 3} and Lemma \ref{Lemma 6}, then $\Phi_\lambda$ has a nontrivial critical point $v_{\lambda,k}$ with $\Phi_\lambda(v_{\lambda,k}) = c_k$ by Proposition \ref{Proposition 7}. An argument similar to that in the proof of Theorem \ref{Theorem 1} shows that $c_k \nearrow \infty$. Since $\Phi_\lambda$ is bounded on bounded sets, then $\norm{v_{\lambda,k}} \to \infty$.

If $c \le 0$, then the equation $- (r - s)\, G(u) + cs = 0$ has no nontrivial solution since $G(u) > 0$ for all $u \in W \setminus \set{0}$. So $\N_c = \emptyset$ by \eqref{23} and hence $\Phi_\lambda$ has no nontrivial critical point with critical value $c$ by Lemma \ref{Lemma 8}.
\end{proof}

\begin{figure}[H]
	\centering
	\begin{tikzpicture}[>=latex]
		%x axis
		\draw[->] (-1,0) -- (6,0) node[below] {\scalebox{0.8}{$\lambda$}};
		\foreach \x in {}
		\draw[shift={(\x,0)}] (0pt,2pt) -- (0pt,-2pt) node[below] {\footnotesize $\x$};
		%y axis
		\draw[->] (0,-.5) -- (0,3) node[left] {\scalebox{0.8}{$c$}};
		\foreach \y in {}
		\draw[shift={(0,\y)}] (2pt,0pt) -- (-2pt,0pt) node[left] {\footnotesize $\y$};
		%\node[below left] at (0,0) {\footnotesize $0$};
		%	\draw[red,thick] (0,-2) .. controls (0,-1.5) and (0,0) .. (4,1.5);
		\draw[blue,thick] (-2.5,3) .. controls (-1,0) and (1,0) .. (1,0);
		%\draw[blue,thick] (2,0) .. controls (3,-1.4) .. (3,-1.4);
		%\draw[blue,thick,dashed] (0,1) .. controls (0,1) and (1,1) .. (2,0);
		\draw[blue,thick] (-1.7,3) .. controls (0,0) and (1.5,0) .. (2,0);
		%\draw[blue,thick] (3,0) .. controls (4,-1.4) .. (4,-1.4);
		\draw [thick] (1,-.1) node[below]{\scalebox{0.8}{$\lambda_{1}$}} -- (1,0.05); 
		\draw [thick] (-.4,0) -- (-.4,3);
		
		\draw [thick] (2,-.1) node[below]{\scalebox{0.8}{$\lambda_{2}$}} -- (2,0.05); 
		\draw [thick] (3,-.1) node[below]{\scalebox{0.8}{$\lambda_{3}$}} -- (3,0.05); 
		\draw [thick] (4,-.1) node[below]{\scalebox{0.8}{$\lambda_{k}$}} -- (4,0.05); 
		
		\draw [thick] (.9,1) node[above]{$\cdots$}; 
		\draw[blue,thick] (-0.9,3) .. controls (1,0) and (2,0) .. (3,0);
		\draw[blue,thick] (-0.1,3) .. controls (1.8,0) and (3,0) .. (4,0);
		
		%\draw[blue,thick] (4,0) .. controls (5,-1.4) .. (5,-1.4);
		%\draw [thick] (3.2,-1.3) node[above]{$\vdots$} -- (3.2,-1.3); 
		\draw [thick] (-.1,2.3) node[right]{$\vdots$} ;%-- (-.1,2,2);
		%\draw (.5,.35) node{$\bullet$};
		\draw (-.4,.5) node{\scalebox{0.8}{$\bullet$}};
		\draw (-.4,1.15) node{\scalebox{0.8}{$\bullet$}};
		\draw (-.4,2.25) node{\scalebox{0.8}{$\bullet$}};
		%\draw (1.8,-.18) node{\scalebox{0.8}{$\bullet$}};
		%\draw (1.8,-.37) node{\scalebox{0.8}{$\bullet$}};
		%\draw (1.8,-.82) node{\scalebox{0.8}{$\bullet$}};
		\draw  (-2,.1) node[above]{\scalebox{1.5}{$\nexists$}}  ; 	
		\draw  (-2.8,2.5) node[above]{\scalebox{0.8}{$\lambda_{c,1}$}} ; 	
		\draw  (-1.9,2.5) node[above]{\scalebox{0.8}{$\lambda_{c,2}$}} ; 
		\draw  (-1.1,2.5) node[above]{\scalebox{0.8}{$\lambda_{c,3}$}} ; 
		
		\draw  (.4,2.5) node[above]{\scalebox{0.8}{$\lambda_{c,k}$}} ; 
		
		%\draw [thick] (-2.8,1.1) node[above]{\scalebox{0.8}{$\mu_{1,c}$}} -- (-2.8,1.2); 	
		%	\draw[blue,thick,dashed] (3,1.3) .. controls (3,1.3) and (4,1.5) .. (4,1.5);
		%\draw [thick] (1,-.1) node[below]{$\mu_1$} -- (1,0.1); 
		%\draw [thick] (2,-.05) node[above]{\scalebox{0.8}{$\mu_{1,0}$}} -- (2,0.1); 
		%\draw [thick] (3,-.05) node[below]{\scalebox{0.8}{$\mu_2$}} -- (3,0.1); 
		%\draw [thick] (4,-.05) node[below]{\scalebox{0.8}{$\mu_n$}} -- (4,.1); 
		%\draw [thick] (1.8,0) node[above left]{\scalebox{0.8}{\tiny$\overline{\mu}$}} -- (1.8,0); 
		%\draw [thick] (2.5,-0.05) node[above]{\scalebox{0.8}{$\mu^*$}} -- (2.5,0.1); 
		%	\draw [thick] (-.1,-2) node[left]{$-\infty$} -- (.1,-2); 
		%	\node[] at (1,1.5) { {\color{blue}$\Phi_\mu(w_\mu)$}};
		%	\node[] at (3,0.5) { {\color{red}$\Phi_\mu(u_\mu)$}};
		%\draw  (.8,0.1) node[above]{\scalebox{0.8}{$\Phi_\mu(u_\mu)$}} ; 
		%\draw  (2.5,-2) node[above]{\scalebox{0.8}{$\Phi_\mu(v_\mu)$}} ; 
		\draw  (-.5,0) node[below]{\scalebox{0.8}{$L_\lambda$}} ;
	\end{tikzpicture}
	\caption{Energy curves for the case $F=0,\, G>0$ (see Theorem \ref{Theorem 3}).} \label{fig3}
\end{figure}

\begin{theorem} \label{Theorem 4}
Assume $(A_1)$--$(A_{11})$, $(H_1)$--$(H_3)$, $(H_4)(iv)$, and $(H_5)$.
\begin{enumroman}
\item For each $c \in I^-$, $\lambda_{c,k} \nearrow \infty$ is a sequence of critical values of $\lambda_c$. Moreover, $\Phi_{\lambda_{c,k}}$ has a nontrivial critical point $u_{c,k}$ with $\Phi_{\lambda_{c,k}}(u_{c,k}) = c$.
\item For each $k \ge 1$, the mapping $I^- \to \R,\, c \mapsto \lambda_{c,k}$ is continuous, decreasing, and satisfies
    \[
    \lim_{c \to 0^-}\, \lambda_{c,k} = \lambda_k, \qquad \lim_{c \to - \infty}\, \lambda_{c,k} = \infty.
    \]
\item For $\lambda \le \lambda_1$, $\Phi_\lambda$ has no nontrivial critical point. For each $\lambda > \lambda_k$, $\Phi_\lambda$ has at least $k$ pairs of nontrivial critical points. For all $\lambda \in \R$, $\Phi_\lambda$ has no nontrivial critical point with critical value in $\R \setminus I^-$.
\end{enumroman}
\end{theorem}

\begin{proof}
Similar to the proof of Theorem \ref{Theorem 2}.
\end{proof}

\begin{figure}[h!]
	\centering
	\begin{tikzpicture}[>=latex]
		%x axis
		\draw[->] (-1,0) -- (6,0) node[below] {\scalebox{0.8}{$\lambda$}};
		\foreach \x in {}
		\draw[shift={(\x,0)}] (0pt,2pt) -- (0pt,-2pt) node[below] {\footnotesize $\x$};
		%y axis
		\draw[->] (0,-3) -- (0,1) node[left] {\scalebox{0.8}{$c$}};
		\foreach \y in {}
		\draw[shift={(0,\y)}] (2pt,0pt) -- (-2pt,0pt) node[left] {\footnotesize $\y$};
		%\node[below left] at (0,0) {\footnotesize $0$};
		%	\draw[red,thick] (0,-2) .. controls (0,-1.5) and (0,0) .. (4,1.5);
		\draw[blue,thick] (.5,0) .. controls (1,0) and (2,0) .. (4.5,-3);
		%\draw[blue,thick] (2,0) .. controls (3,-1.4) .. (3,-1.4);
		%\draw[blue,thick,dashed] (0,1) .. controls (0,1) and (1,1) .. (2,0);
		\draw[blue,thick] (1.5,0) .. controls (2,0) and (3,0) .. (5,-3);
		%\draw[blue,thick] (3,0) .. controls (4,-1.4) .. (4,-1.4);
		\draw [thick] (.5,-.1) node[above]{\scalebox{0.8}{$\lambda_{1}$}} -- (.5,0.05); 
		\draw [thick] (3,0) -- (3,-3);
		
		\draw [thick] (1.5,-.1) node[above]{\scalebox{0.8}{$\lambda_{2}$}} -- (1.5,.05); 
		\draw [thick] (2.5,-.1) node[above]{\scalebox{0.8}{$\lambda_{3}$}} -- (2.5,0.05); 
		\draw [thick] (3.5,-.1) node[above]{\scalebox{0.8}{$\lambda_{k}$}} -- (3.5,0.05); 
		
		\draw [thick] (4.9,-1.5) node[above]{$\cdots$}; 
		\draw[blue,thick] (2.5,0) .. controls (3,0) and (4,0) .. (5.5,-3);
		\draw[blue,thick] (3.5,0) .. controls (4,0) and (5,0) .. (6,-3);
		
		%\draw[blue,thick] (4,0) .. controls (5,-1.4) .. (5,-1.4);
		%\draw [thick] (3.2,-1.3) node[above]{$\vdots$} -- (3.2,-1.3); 
		%\draw [thick] (-.1,2.3) node[right]{$\vdots$} ;%-- (-.1,2,2);
		%\draw (.5,.35) node{$\bullet$};
		\draw (3,-.1) node{\scalebox{0.8}{$\bullet$}};
		\draw (3,-.6) node{\scalebox{0.8}{$\bullet$}};
		\draw (3,-1.4) node{\scalebox{0.8}{$\bullet$}};
	%	\draw (4.5,-3) node{\scalebox{0.8}{$\bullet$}};
		%\draw (1.8,-.18) node{\scalebox{0.8}{$\bullet$}};
		%\draw (1.8,-.37) node{\scalebox{0.8}{$\bullet$}};
		%\draw (1.8,-.82) node{\scalebox{0.8}{$\bullet$}};
		\draw  (1,-3) node[above]{\scalebox{1.5}{$\nexists$}}  ; 	
		\draw  (6.5,-3.3) node[above]{\scalebox{0.8}{$\lambda_{c,k}$}} ; 	
	%	\draw  (6.5,-2.3) node[above]{\scalebox{0.8}{$\lambda_{c,3}$}} ; 
	%	\draw  (6.5,-2.8) node[above]{\scalebox{0.8}{$\lambda_{c,2}$}} ; 
		
		\draw  (4,-3.3) node[above]{\scalebox{0.8}{$\lambda_{c,1}$}} ; 
		
		%\draw [thick] (-2.8,1.1) node[above]{\scalebox{0.8}{$\mu_{1,c}$}} -- (-2.8,1.2); 	
		%	\draw[blue,thick,dashed] (3,1.3) .. controls (3,1.3) and (4,1.5) .. (4,1.5);
		%\draw [thick] (1,-.1) node[below]{$\mu_1$} -- (1,0.1); 
		%\draw [thick] (2,-.05) node[above]{\scalebox{0.8}{$\mu_{1,0}$}} -- (2,0.1); 
		%\draw [thick] (3,-.05) node[below]{\scalebox{0.8}{$\mu_2$}} -- (3,0.1); 
		%\draw [thick] (4,-.05) node[below]{\scalebox{0.8}{$\mu_n$}} -- (4,.1); 
		%\draw [thick] (1.8,0) node[above left]{\scalebox{0.8}{\tiny$\overline{\mu}$}} -- (1.8,0); 
		%\draw [thick] (2.5,-0.05) node[above]{\scalebox{0.8}{$\mu^*$}} -- (2.5,0.1); 
		%	\draw [thick] (-.1,-2) node[left]{$-\infty$} -- (.1,-2); 
		%	\node[] at (1,1.5) { {\color{blue}$\Phi_\mu(w_\mu)$}};
		%	\node[] at (3,0.5) { {\color{red}$\Phi_\mu(u_\mu)$}};
		%\draw  (.8,0.1) node[above]{\scalebox{0.8}{$\Phi_\mu(u_\mu)$}} ; 
		%\draw  (2.5,-2) node[above]{\scalebox{0.8}{$\Phi_\mu(v_\mu)$}} ; 
		\draw  (3,-3) node[left]{\scalebox{0.8}{$L_\lambda$}} ;
	\end{tikzpicture}
	\caption{Energy curves for the case $F=0,\, G<0$ (see Theorem \ref{Theorem 4}).} \label{fig4}
\end{figure}

\begin{theorem} \label{Theorem 5}
Assume $(A_1)$--$(A_{11})$, $(H_1)$--$(H_3)$, $(H_4)(v)$, and $(H_5)$.
\begin{enumroman}
\item For each $c \in I^-$, $\lambda_{c,k} \nearrow \infty$ is a sequence of critical values of $\lambda_c$. Moreover, $\Phi_{\lambda_{c,k}}$ has a nontrivial critical point $u_{c,k}$ with $\Phi_{\lambda_{c,k}}(u_{c,k}) = c$.
\item For each $k \ge 1$, the mapping $I^- \to \R,\, c \mapsto \lambda_{c,k}$ is continuous, decreasing, and satisfies
    \[
    \lim_{c \to 0^-}\, \lambda_{c,k} = - \infty, \qquad \lim_{c \to - \infty}\, \lambda_{c,k} = \infty.
    \]
\item For each $\lambda \in \R$, $\Phi_\lambda$ has a sequence of nontrivial critical points $\seq{v_{\lambda,k}}$ such that $\Phi_\lambda(v_{\lambda,k}) \nearrow 0$ and $\norm{v_{\lambda,k}} \to 0$. Moreover, $\Phi_\lambda$ has no nontrivial critical point with critical value in $\R \setminus I^-$.
\end{enumroman}
\end{theorem}

\begin{proof}
Similar to the proof of Theorem \ref{Theorem 1}.
\end{proof}

\begin{figure}[h!]
	\centering
	\begin{tikzpicture}[>=latex]
		%x axis
		\draw[->] (-1,0) -- (6,0) node[below] {\scalebox{0.8}{$\lambda$}};
		\foreach \x in {}
		\draw[shift={(\x,0)}] (0pt,2pt) -- (0pt,-2pt) node[below] {\footnotesize $\x$};
		%y axis
		\draw[->] (0,-3) -- (0,1) node[left] {\scalebox{0.8}{$c$}};
		\foreach \y in {}
		\draw[shift={(0,\y)}] (2pt,0pt) -- (-2pt,0pt) node[left] {\footnotesize $\y$};
		%\node[below left] at (0,0) {\footnotesize $0$};
		%	\draw[red,thick] (0,-2) .. controls (0,-1.5) and (0,0) .. (4,1.5);
		\draw[blue,thick] (-1,-.4) .. controls (1.5,-.5) and (2.5,-.5) .. (4.5,-3);
		%\draw[blue,thick] (2,0) .. controls (3,-1.4) .. (3,-1.4);
		%\draw[blue,thick,dashed] (0,1) .. controls (0,1) and (1,1) .. (2,0);
		\draw[blue,thick] (-1,-.3) .. controls (2,-.4) and (3,-.4) .. (5,-3);
		%\draw[blue,thick] (3,0) .. controls (4,-1.4) .. (4,-1.4);
		%	\draw [thick] (.5,-.1) node[above]{\scalebox{0.8}{$\lambda_{1}$}} -- (.5,0.05); 
		\draw [thick] (3,0) -- (3,-3);
		
		%	\draw [thick] (1.5,-.1) node[above]{\scalebox{0.8}{$\lambda_{2}$}} -- (1.5,.05); 
		%	\draw [thick] (2.5,-.1) node[above]{\scalebox{0.8}{$\lambda_{3}$}} -- (2.5,0.05); 
		%	\draw [thick] (3.5,-.1) node[above]{\scalebox{0.8}{$\lambda_{k}$}} -- (3.5,0.05); 
		
		\draw [thick] (4.8,-2.2) node[above]{$\cdots$}; 
		\draw[blue,thick] (-1,-.2) .. controls (2.5,-.3) and (3,-.3) .. (5.5,-3);
		\draw[blue,thick] (-1,-.1) .. controls (3,-.2) and (3.5,-.2) .. (6,-3);
		
		%\draw[blue,thick] (4,0) .. controls (5,-1.4) .. (5,-1.4);
		%\draw [thick] (3.2,-1.3) node[above]{$\vdots$} -- (3.2,-1.3); 
		%\draw [thick] (-.1,2.3) node[right]{$\vdots$} ;%-- (-.1,2,2);
		%\draw (.5,.35) node{$\bullet$};
		\draw (3,-.55) node{\scalebox{0.8}{$\bullet$}};
		\draw (3,-.8) node{\scalebox{0.8}{$\bullet$}};
		\draw (3,-1.05) node{\scalebox{0.8}{$\bullet$}};
		\draw (3,-1.45) node{\scalebox{0.8}{$\bullet$}};
		%\draw (1.8,-.18) node{\scalebox{0.8}{$\bullet$}};
		%\draw (1.8,-.37) node{\scalebox{0.8}{$\bullet$}};
		%\draw (1.8,-.82) node{\scalebox{0.8}{$\bullet$}};
		\draw  (1,-3) node[above]{\scalebox{1.5}{$\nexists$}}  ; 	
		\draw  (6.5,-3.3) node[above]{\scalebox{0.8}{$\lambda_{c,k}$}} ; 	
		%	\draw  (6.5,-2.3) node[above]{\scalebox{0.8}{$\lambda_{c,3}$}} ; 
		%	\draw  (6.5,-2.8) node[above]{\scalebox{0.8}{$\lambda_{c,2}$}} ; 
		
		\draw  (4,-3.3) node[above]{\scalebox{0.8}{$\lambda_{c,1}$}} ; 
		
		%\draw [thick] (-2.8,1.1) node[above]{\scalebox{0.8}{$\mu_{1,c}$}} -- (-2.8,1.2); 	
		%	\draw[blue,thick,dashed] (3,1.3) .. controls (3,1.3) and (4,1.5) .. (4,1.5);
		%\draw [thick] (1,-.1) node[below]{$\mu_1$} -- (1,0.1); 
		%\draw [thick] (2,-.05) node[above]{\scalebox{0.8}{$\mu_{1,0}$}} -- (2,0.1); 
		%\draw [thick] (3,-.05) node[below]{\scalebox{0.8}{$\mu_2$}} -- (3,0.1); 
		%\draw [thick] (4,-.05) node[below]{\scalebox{0.8}{$\mu_n$}} -- (4,.1); 
		%\draw [thick] (1.8,0) node[above left]{\scalebox{0.8}{\tiny$\overline{\mu}$}} -- (1.8,0); 
		%\draw [thick] (2.5,-0.05) node[above]{\scalebox{0.8}{$\mu^*$}} -- (2.5,0.1); 
		%	\draw [thick] (-.1,-2) node[left]{$-\infty$} -- (.1,-2); 
		%	\node[] at (1,1.5) { {\color{blue}$\Phi_\mu(w_\mu)$}};
		%	\node[] at (3,0.5) { {\color{red}$\Phi_\mu(u_\mu)$}};
		%\draw  (.8,0.1) node[above]{\scalebox{0.8}{$\Phi_\mu(u_\mu)$}} ; 
		%\draw  (2.5,-2) node[above]{\scalebox{0.8}{$\Phi_\mu(v_\mu)$}} ; 
		\draw  (3,0) node[above]{\scalebox{0.8}{$L_\lambda$}} ;
	\end{tikzpicture}
	\caption{Energy curves for the case $F>0,\, G<0$ (see Theorem \ref{Theorem 5}).} \label{fig5}
\end{figure}

\begin{theorem} \label{Theorem 6}
Assume $(A_1)$--$(A_{11})$, $(H_1)$--$(H_3)$, $(H_4)(vi)$, and $(H_5)$.
\begin{enumroman}
\item For each $c \in I^+$, $\lambda_{c,k} \nearrow \infty$ is a sequence of critical values of $\lambda_c$. Moreover, $\Phi_{\lambda_{c,k}}$ has a nontrivial critical point $u_{c,k}$ with $\Phi_{\lambda_{c,k}}(u_{c,k}) = c$.
\item For each $k \ge 1$, the mapping $I^+ \to \R,\, c \mapsto \lambda_{c,k}$ is continuous, decreasing, and satisfies
    \[
    \lim_{c \to 0^+}\, \lambda_{c,k} = \infty, \qquad \lim_{c \to \infty}\, \lambda_{c,k} = - \infty.
    \]
\item For each $\lambda \in \R$, $\Phi_\lambda$ has a sequence of nontrivial critical points $\seq{v_{\lambda,k}}$ such that $\Phi_\lambda(v_{\lambda,k}) \nearrow \infty$ and $\norm{v_{\lambda,k}} \to \infty$. Moreover, $\Phi_\lambda$ has no nontrivial critical point with critical value in $\R \setminus I^+$.
\end{enumroman}
\end{theorem}

\begin{proof}
Similar to the proof of Theorem \ref{Theorem 3}.
\end{proof}

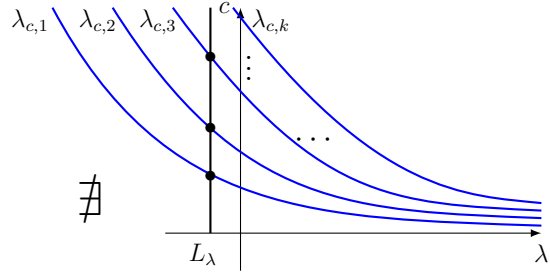
\begin{figure}[H]
	\centering
	\begin{tikzpicture}[>=latex]
		%x axis
		\draw[->] (-1,0) -- (4,0) node[below] {\scalebox{0.8}{$\lambda$}};
		\foreach \x in {}
		\draw[shift={(\x,0)}] (0pt,2pt) -- (0pt,-2pt) node[below] {\footnotesize $\x$};
		%y axis
		\draw[->] (0,-.5) -- (0,3) node[left] {\scalebox{0.8}{$c$}};
		\foreach \y in {}
		\draw[shift={(0,\y)}] (2pt,0pt) -- (-2pt,0pt) node[left] {\footnotesize $\y$};
		%\node[below left] at (0,0) {\footnotesize $0$};
		%	\draw[red,thick] (0,-2) .. controls (0,-1.5) and (0,0) .. (4,1.5);
		\draw[blue,thick] (-2.5,3) .. controls (-1,.2) and (1,.2) .. (4,.1);
		%\draw[blue,thick] (2,0) .. controls (3,-1.4) .. (3,-1.4);
		%\draw[blue,thick,dashed] (0,1) .. controls (0,1) and (1,1) .. (2,0);
		\draw[blue,thick] (-1.7,3) .. controls (0,.3) and (1.5,.3) .. (4,.2);
		%\draw[blue,thick] (3,0) .. controls (4,-1.4) .. (4,-1.4);
	%	\draw [thick] (1,-.1) node[below]{\scalebox{0.8}{$\lambda_{1}$}} -- (1,0.05); 
		\draw [thick] (-.4,0) -- (-.4,3);
		
	%	\draw [thick] (2,-.1) node[below]{\scalebox{0.8}{$\lambda_{2}$}} -- (2,0.05); 
	%	\draw [thick] (3,-.1) node[below]{\scalebox{0.8}{$\lambda_{3}$}} -- (3,0.05); 
	%	\draw [thick] (4,-.1) node[below]{\scalebox{0.8}{$\lambda_{k}$}} -- (4,0.05); 
		
		\draw [thick] (1,1) node[above]{$\cdots$}; 
		\draw[blue,thick] (-0.9,3) .. controls (1,.4) and (2,.4) .. (4,.3);
		\draw[blue,thick] (-0.1,3) .. controls (1.8,.5) and (3,.5) .. (4,.4);
		
		%\draw[blue,thick] (4,0) .. controls (5,-1.4) .. (5,-1.4);
		%\draw [thick] (3.2,-1.3) node[above]{$\vdots$} -- (3.2,-1.3); 
		\draw [thick] (-.1,2.3) node[right]{$\vdots$} ;%-- (-.1,2,2);
		%\draw (.5,.35) node{$\bullet$};
		\draw (-.4,.75) node{\scalebox{0.8}{$\bullet$}};
		\draw (-.4,1.4) node{\scalebox{0.8}{$\bullet$}};
		\draw (-.4,2.34) node{\scalebox{0.8}{$\bullet$}};
		%\draw (1.8,-.18) node{\scalebox{0.8}{$\bullet$}};
		%\draw (1.8,-.37) node{\scalebox{0.8}{$\bullet$}};
		%\draw (1.8,-.82) node{\scalebox{0.8}{$\bullet$}};
		\draw  (-2,.1) node[above]{\scalebox{1.5}{$\nexists$}}  ; 	
		\draw  (-2.8,2.5) node[above]{\scalebox{0.8}{$\lambda_{c,1}$}} ; 	
		\draw  (-1.9,2.5) node[above]{\scalebox{0.8}{$\lambda_{c,2}$}} ; 
		\draw  (-1.1,2.5) node[above]{\scalebox{0.8}{$\lambda_{c,3}$}} ; 
		
 			\draw  (.4,2.5) node[above]{\scalebox{0.8}{$\lambda_{c,k}$}} ; 
		
		%\draw [thick] (-2.8,1.1) node[above]{\scalebox{0.8}{$\mu_{1,c}$}} -- (-2.8,1.2); 	
		%	\draw[blue,thick,dashed] (3,1.3) .. controls (3,1.3) and (4,1.5) .. (4,1.5);
		%\draw [thick] (1,-.1) node[below]{$\mu_1$} -- (1,0.1); 
		%\draw [thick] (2,-.05) node[above]{\scalebox{0.8}{$\mu_{1,0}$}} -- (2,0.1); 
		%\draw [thick] (3,-.05) node[below]{\scalebox{0.8}{$\mu_2$}} -- (3,0.1); 
		%\draw [thick] (4,-.05) node[below]{\scalebox{0.8}{$\mu_n$}} -- (4,.1); 
		%\draw [thick] (1.8,0) node[above left]{\scalebox{0.8}{\tiny$\overline{\mu}$}} -- (1.8,0); 
		%\draw [thick] (2.5,-0.05) node[above]{\scalebox{0.8}{$\mu^*$}} -- (2.5,0.1); 
		%	\draw [thick] (-.1,-2) node[left]{$-\infty$} -- (.1,-2); 
		%	\node[] at (1,1.5) { {\color{blue}$\Phi_\mu(w_\mu)$}};
		%	\node[] at (3,0.5) { {\color{red}$\Phi_\mu(u_\mu)$}};
		%\draw  (.8,0.1) node[above]{\scalebox{0.8}{$\Phi_\mu(u_\mu)$}} ; 
		%\draw  (2.5,-2) node[above]{\scalebox{0.8}{$\Phi_\mu(v_\mu)$}} ; 
		\draw  (-.5,0) node[below]{\scalebox{0.8}{$L_\lambda$}} ;
	\end{tikzpicture}
	\caption{Energy curves for the case $F<0,\, G>0$ (see Theorem \ref{Theorem 6}).} \label{fig6}
\end{figure}

\section{Proofs of Theorems \ref{Theorem 101}--\ref{Theorem 106}} \label{Section 3}

In this section we will show that equations \eqref{103}--\eqref{108} fit into the abstract framework of the last section. Theorems \ref{Theorem 101}--\ref{Theorem 106} will then follow from Theorems \ref{Theorem 1}--\ref{Theorem 6}, respectively.

We take $W = E_r(\R^3)$, $u_t = t^2\, u(t\, \cdot)$, and $s = 3$. It was shown in Mercuri and Perera \cite[Lemma 3.1]{MePe2} that the mapping $E_r(\R^3) \times [0,\infty) \to E_r(\R^3),\, (u,t) \mapsto u_t$ is continuous. Assumptions $(A_1)$--$(A_3)$ are clearly satisfied, while $(A_4)$ and $(A_5)$ follow from the estimate
\[
\norm{u_t} = \left[t^3 \int_{\R^3} |\nabla u|^2\, dx + t^{3/2} \left(\int_{\R^3} \int_{\R^3} \frac{u^2(x)\, u^2(y)}{|x - y|}\, dx\, dy\right)^{1/2}\right]^{1/2} \le \max \set{t^{3/2},t^{3/4}} \norm{u}.
\]

The operators $\As$ and $\Bs$ given by
\begin{multline*}
\As(u) v = \int_{\R^3} \nabla u \cdot \nabla v\, dx + \frac{1}{4 \pi} \int_{\R^3} \int_{\R^3} \frac{u^2(x)\, u(y)\, v(y)}{|x - y|}\, dx\, dy, \quad \Bs(u) v = \int_{\R^3} |u| uv\, dx,\\[7.5pt]
u, v \in E_r(\R^3)
\end{multline*}
are scaled operators that clearly satisfy $(A_6)$ and $(A_8)$. Assumption $(A_7)$ is verified in Mercuri and Perera \cite[Lemma 3.2]{MePe2}, while $(A_9)$ follows from the compactness of the embedding $E_r(\R^3) \hookrightarrow L^3(\R^3)$. The potentials $I$ and $J$ of $\As$ and $\Bs$ given in \eqref{200} clearly satisfy $(A_{10})$, and $(A_{11})$ is verified in Mercuri and Perera \cite[Section 3.1]{MePe2}.

In Theorems \ref{Theorem 101}--\ref{Theorem 106}, the operators $f$ and $g$ are given by
\begin{enumerate}
\item[$(i)$] $f(u) v = \dint_{\R^3} |u|^{\sigma - 2}\, uv\, dx$ and $g = 0$,
\item[$(ii)$] $f(u) v = - \dint_{\R^3} |u|^{\sigma - 2}\, uv\, dx$ and $g = 0$,
\item[$(iii)$] $f = 0$ and $g(u) v = \dint_{\R^3} |u|^{\tau - 2}\, uv\, dx$,
\item[$(iv)$] $f = 0$ and $g(u) v = - \dint_{\R^3} |u|^{\tau - 2}\, uv\, dx$,
\item[$(v)$] $f(u) v = \dint_{\R^3} |u|^{\sigma - 2}\, uv\, dx$ and $g(u) v = - \dint_{\R^3} |u|^{\tau - 2}\, uv\, dx$,
\item[$(vi)$] $f(u) v = - \dint_{\R^3} |u|^{\sigma - 2}\, uv\, dx$ and $g(u) v = \dint_{\R^3} |u|^{\tau - 2}\, uv\, dx$,
\end{enumerate}
respectively. Assumption $(H_1)$ is satisfied with $q = 2 \sigma - 3$ and $r = 2 \tau - 3$, while $(H_2)$ follows from the compactness of the embedding $E_r(\R^3) \hookrightarrow L^\rho(\R^3)$ for $\rho \in (18/7,6)$. The potentials $F$ and $G$ of $f$ and $g$ clearly satisfy $(i)$--$(vi)$ of $(H_4)$ in the above six cases, and $(H_5)$ is satisfied since
\[
h(u) u = 2\, [(3 - \sigma)\, f(u) u - (\tau - 3)\, g(u) u] \ne 0 \quad \forall u \in E_r(\R^3) \setminus \set{0}
\]
in all cases.

It only remains to verify $(H_3)$, which we do in the following lemma.

\begin{lemma}
For any $\alpha, \beta, \gamma, \delta \in \R$, every solution of the equation
\begin{equation} \label{300}
\alpha \left[- \Delta u + \left(\frac{1}{4 \pi |x|} \star u^2\right) u\right] = \beta\, |u| u + \gamma\, |u|^{\sigma - 2}\, u + \delta\, |u|^{\tau - 2}\, u \quad \text{in } \R^3
\end{equation}
satisfies the Poho\v{z}aev type identity
\begin{multline*}
\alpha \left[\frac{1}{2} \int_{\R^3} |\nabla u|^2\, dx + \frac{1}{16 \pi} \int_{\R^3} \int_{\R^3} \frac{u^2(x)\, u^2(y)}{|x - y|}\, dx\, dy\right] = \frac{\beta}{3} \int_{\R^3} |u|^3\, dx\\[7.5pt]
+ \left(\frac{2}{3} - \frac{1}{\sigma}\right) \gamma \int_{\R^3} |u|^\sigma\, dx + \left(\frac{2}{3} - \frac{1}{\tau}\right) \delta \int_{\R^3} |u|^\tau\, dx.
\end{multline*}
\end{lemma}

\begin{proof}
If $u \in E_r(\R^3)$ is a solution of the equation \eqref{300}, then testing it with $u$ itself gives
\begin{multline} \label{62}
\alpha \left[\int_{\R^3} |\nabla u|^2\, dx + \frac{1}{4 \pi} \int_{\R^3} \int_{\R^3} \frac{u^2(x)\, u^2(y)}{|x - y|}\, dx\, dy\right] = \beta \int_{\R^3} |u|^3\, dx\\[7.5pt]
+ \gamma \int_{\R^3} |u|^\sigma\, dx + \delta \int_{\R^3} |u|^\tau\, dx,
\end{multline}
and $u$ also satisfies the Poho\v{z}aev type identity
\begin{multline} \label{63}
\alpha \left[\frac{1}{2} \int_{\R^3} |\nabla u|^2\, dx + \frac{5}{16 \pi} \int_{\R^3} \int_{\R^3} \frac{u^2(x)\, u^2(y)}{|x - y|}\, dx\, dy\right] = \beta \int_{\R^3} |u|^3\, dx\\[7.5pt]
+ \frac{3 \gamma}{\sigma} \int_{\R^3} |u|^\sigma\, dx + \frac{3 \delta}{\tau} \int_{\R^3} |u|^\tau\, dx
\end{multline}
(see, e.g., Ianni and Ruiz \cite[Proposition 2.5]{MR2902293}). Multiplying \eqref{62} by $2/3$ and \eqref{63} by $1/3$ and subtracting gives the desired conclusion.
\end{proof}

{\bf Acknowledgement.}
This work was completed while the second author held a post-doctoral position at Florida Institute of Technology, Melbourne, United States of
America, supported  by CNPq/Brazil under Grant 201334/2024-0.

\def\cprime{$''$}

\end{document}